\documentclass[a4paper]{amsart}

\usepackage{amsmath}
\usepackage{amsfonts}
\usepackage{amssymb}
\usepackage{amsthm}
\usepackage{eucal}
\usepackage{upgreek}

\numberwithin{equation}{section}

\theoremstyle{plain}
\newtheorem{theorem}{Theorem}[section]

\newtheorem{lemma}[theorem]{Lemma}
\newtheorem{proposition}[theorem]{Proposition}

\theoremstyle{definition}
\newtheorem{definition}[theorem]{Definition}
\newtheorem{example}[theorem]{Example}
\theoremstyle{remark}
\newtheorem{remark}[theorem]{Remark}
\newcounter{counter_a}
\newenvironment{myenum}{\begin{list}{{\rm(\roman{counter_a})}}%
{\usecounter{counter_a}
\setlength{\itemsep}{0.5ex}\setlength{\topsep}{1.1ex}
\setlength{\leftmargin}{5ex}\setlength{\labelwidth}{5ex}}}{\end{list}}
\newcounter{counter_b}

\DeclareMathOperator\Real{Re}

\renewcommand\Re{\Real}

\newcommand\cD{\mathcal D}

\newcommand\cH{\mathcal H}
\newcommand\cK{\mathcal K}

\newcommand\cL{\mathcal L}
\newcommand\cM{\mathcal M}

\newcommand\cV{\mathcal V}
\newcommand\CC{\mathbb C}
\newcommand\NN{\mathbb N}
\newcommand\RR{\mathbb R}

\newcommand\fra{\mathfrak a}
\newcommand\frb{\mathfrak b}
\newcommand\frc{\mathfrak c}
\newcommand\frd{\mathfrak d}
\newcommand\frh{\mathfrak h}

\newcommand\frt{\mathfrak t}
\newcommand\frv{\mathfrak v}
\newcommand\eps{\varepsilon}

\newcommand\ov{\overline}

\newcommand\rd{\mathrm{d}}
\newcommand{\defeq}{\mathrel{\mathop:}=}

\newcommand{\defequ}{\mathrel{\mathop:}\hspace*{-0.72ex}&=}

\newcommand\sess{\sigma_{\rm ess}}
\newcommand\sigmadis{\sigma_{\rm dis}}

\newcommand\M{\mathbf{M}}

\newcommand\lae{\lambda_{\rm e}}
\newcommand\VM{{\rm(VM$^-$)}}
\newcommand\VMbf{{\rm\textbf{(VM{\boldmath$^-$}\!)}}}

\DeclareMathOperator\dom{dom}
\DeclareMathOperator\ran{ran}
\DeclareMathOperator\spn{span}
\DeclareMathOperator\diag{diag}

\newcommand\void[1]{}

\begin{document}

\title[Triple variational principles]{Triple variational principles for \\ self-adjoint
operator functions}

\author{Matthias Langer}

\address{Department of Mathematics and Statistics,
University of Strathclyde \newline
26 Richmond Street, Glasgow G1 1XH, United Kingdom}

\author{Michael Strauss}

\address{School of Mathematics, Cardiff University \newline
Senghennydd Road, Cardiff CF24 4AG, United Kingdom}

\begin{abstract}
{For a very general class of unbounded self-adjoint operator function we prove
upper bounds for eigenvalues which lie within arbitrary gaps of the essential spectrum.
These upper bounds are given by triple variations.
Furthermore, we find conditions which imply that a point is in the resolvent set.
For norm resolvent continuous operator functions we show that the variational
inequality becomes an equality.
\\[1ex]
\\[1ex]
}
\end{abstract}

\maketitle

\section{Introduction}

In many applications of operator and spectral theory eigenvalue problems appear
which are nonlinear in the eigenvalue parameter, e.g.\ polynomially or rationally.
Very often such problems can be dealt with by introducing a function of the
spectral parameter whose values are linear operators in a Hilbert space.
To be more specific, let $T(\cdot)$ be an operator function that is defined on some
set $\Delta\subset\CC$ and whose values are closed linear operators in a
Hilbert space $(\cH,\langle\cdot\,,\cdot\rangle)$; for each $\lambda\in\Delta$
the domain of the operator $T(\lambda)$ is denoted by $\dom(T(\lambda))$.
A number $\lambda\in\Delta$ is called an \emph{eigenvalue} of the operator function $T$ if
there exists an $x\in\dom(T(\lambda))\setminus\{0\}$ such that $T(\lambda)x=0$,
i.e.\ $0$ is in the point spectrum of the operator $T(\lambda)$.
The \emph{spectrum}, \emph{essential spectrum}, \emph{discrete spectrum}
and \emph{resolvent set} of $T$ are defined as follows:
\begin{align*}
  \sigma(T) \defequ \bigl\{\lambda\in\Delta: 0\in\sigma(T(\lambda))\bigr\}, \\[1ex]
  \sess(T) \defequ \bigl\{\lambda\in\Delta: 0\in\sess(T(\lambda))\bigr\}
  = \bigl\{\lambda\in\Delta: T(\lambda)\text{ is not Fredholm}\bigr\}, \\[1ex]
  \sigmadis(T) \defequ \sigma(T)\setminus\sess(T), \\[1ex]
  \rho(T) \defequ \bigl\{\lambda\in\Delta: 0\in\rho(T(\lambda))\bigr\};
\end{align*}
note that a closed operator is called Fredholm if the dimension of the kernel and the
(algebraic) co-dimension of the range are finite.  A trivial example of an
operator function is given by $T(\lambda)=A-\lambda I$ where $A$ is a closed operator;
in this case the spectra of the operator function $T$
and the operator $A$ clearly coincide.
More complicated examples are operator polynomials or Schur complements of block
operator matrices; see, e.g.\ \cite{markus_book,tretter_book}
and references therein; see also the survey article \cite{tisseur_meerbergen01} about numerical
methods for eigenvalues of quadratic matrix polynomials.

It is our aim to show spectral enclosures and variational principles for eigenvalues
of operator functions.  In the 1950s R.\,J.~Duffin \cite{duffin55} proved a variational
principle for eigenvalues of certain quadratic matrix polynomials, which was generalised to
infinite-dimensional spaces and more general operator functions in the following decades;
see, e.g.\ \cite{rogers64,turner67,hadeler68,werner71,barsten74,markus_book}.
Basically, the following situation was considered.
Let $T$ be a differentiable function defined on an interval $[\alpha,\beta]$ whose values
are bounded self-adjoint operators in a Hilbert space $\cH$ such that $T(\alpha)\gg0$
(i.e.\ $T(\alpha)$ is uniformly positive) and $T(\beta)\ll0$.
Moreover, for every $x\in\cH\setminus\{0\}$ the
scalar function $\lambda\mapsto\langle T(\lambda)x,x\rangle$ has exactly one zero
in $(\alpha,\beta)$, which we denote by $p(x)$,
and the inequality $\langle T'(p(x))x,x\rangle<0$ holds.
The mapping $x\mapsto p(x)$ is called a generalised Rayleigh functional.
The eigenvalues of $T$ below the essential spectrum of $T$ can accumulate
at most at the bottom of $\sess(T)$; if they are denoted by
$\lambda_1\le\lambda_2\le\cdots$, then they are characterised by the following
variational principle:
\begin{equation}\label{var_princ_old}
  \lambda_n = \min_{\substack{\cL\subset\cH \\[0.2ex] \dim\cL=n}}\;
  \max_{\substack{x\in\cL \\[0.2ex] x\ne 0}}\;p(x)
  = \max_{\substack{\cL\subset\mathcal{H} \\[0.2ex] \dim\cL=n-1}}\;\;
  \min_{\substack{x\in\cH \\[0.2ex] x\perp\cL,~x\ne 0}}\;p(x);
\end{equation}
here $\cL$ denotes finite-dimensional subspaces of $\cH$.
If $T(\lambda)=A-\lambda I$ where $A$ is a bounded self-adjoint operator,
then \eqref{var_princ_old} reduces to the standard variational principle for eigenvalues
of a self-adjoint operator; the generalised Rayleigh functional is then just the
classical Rayleigh quotient: $p(x)=\frac{\langle Ax,x\rangle}{\|x\|^2}$.

In \cite{BEL00} the assumption that $T(\alpha)$ is uniformly positive was
relaxed and replaced by the assumption that the negative spectrum of $T(\lambda)$
consists of only a finite number $\kappa$ of eigenvalues (counted with multiplicities),
in which case $n$ has to be replaced by $n+\kappa$ in the variations over the subspaces;
also the generalised Rayleigh quotient has to be slightly modified
(see Definition~\ref{def:rayleigh_Mg+}\,(i) below);
cf.\ also \cite{VW82,voss15}.
In \cite{EL04} also functions whose values are unbounded operators were allowed;
see also \cite{JLT}.


The main aim of our paper is to remove the assumption of the finiteness of the
negative spectrum of $T(\alpha)$ and to allow also the characterisation of eigenvalues
in gaps of the essential spectrum.  In order to do this, a third variation is needed;
see Theorem~\ref{th:var_equ}, the main result of the paper.
This theorem greatly sharpens and extends \cite[Theorem~2.4]{EL02}, where only an
inequality was shown for operator functions and where it was assumed that the
values are bounded operators (for some quadratic polynomials equality was proved).
As part of the proof of Theorem~\ref{th:var_equ} we also show such an
inequality (Theorem~\ref{th:varinequ}) for a class of operator functions with
less continuity assumptions then needed in Theorem~\ref{th:var_equ}.
To our knowledge the first triple variational principle appeared in \cite{phillips70}
where eigenvalues of positive operators in Krein spaces were characterised;
this was generalised in \cite{textorius74}.

Our second main result, Theorem~\ref{th:resolv}, is connected with the inequality
in Theorem~\ref{th:varinequ} and gives a sufficient condition for points being
in the resolvent set of an operator function.  In Theorem~\ref{th:nonneg_pert}
this is used to obtain the existence of spectral gaps for perturbed self-adjoint operators.
In a forthcoming paper \cite{LS_bom} we will also apply Theorem~\ref{th:resolv}
to prove spectral inclusions for certain block operator matrices.

Let us give a brief synopsis of the paper.
In Section~\ref{sec:var_princ} we state and prove the result about points in the
resolvent set of an operator function (Theorem~\ref{th:resolv}) and the variational
inequality (Theorem~\ref{th:varinequ}).  We should mention that also an inequality
for the essential spectrum is obtained.
In Section~\ref{sec:selfadjoint_op} we consider self-adjoint operators, which need
not be semi-bounded, and prove a variational principle for eigenvalues
in arbitrary gaps of the essential spectrum (Theorem~\ref{th:selfadjointops}).
Moreover, the above mentioned perturbation result for spectral gaps is proved
there (Theorem~\ref{th:nonneg_pert}).
These results are applied to Dirac operators and to Schr\"odinger operators
with perturbed periodic potentials.
In Section~\ref{sec:decomp} we prove a decomposition of the Hilbert space into a
direct sum of three subspaces, one being the span of the eigenvectors corresponding
to eigenvalues in an interval and the other two being spectral subspaces connected
with the operators at the two endpoints of the interval (Theorem~\ref{th:decomp}).
This is the main ingredient in the proof of the other inequality of the variational
principle in Theorem~\ref{th:var_equ}.  Further, in Section~\ref{sec:decomp} we
prove that eigenvalues cannot accumulate outside the essential spectrum of an
analytic operator function (Proposition~\ref{pr:holom}).
Finally, in Section~\ref{sec:var_equ} we prove the triple variational principle
for eigenvalues of norm resolvent continuous operator functions.
The result is illustrated with a quadratic operator polynomial.

Throughout this paper the term `subspace' refers to a linear manifold, which is not
necessarily closed.
Moreover, $\dotplus$ denotes a direct sum of two subspaces.

\section{A general variational inequality}\label{sec:var_princ}

In this section we consider a rather general class of self-adjoint
operator functions and prove variational inequalities for eigenvalues.
Moreover, we give sufficient conditions for points to belong to the resolvent
set of such operator functions.

Let $A$ be a self-adjoint operator in a Hilbert space $\cH$ and let $E$ be
its spectral measure.
We define the corresponding sesquilinear form $\fra$ by
\begin{equation}\label{defform}
  \fra[x,y] \defeq \int_\RR \mu\, \rd\langle E(\mu)x,y\rangle
\end{equation}
for $x,y\in\dom(\fra)\defeq\dom\bigl(|A|^{1/2}\bigr)$.
Moreover, we introduce the quadratic form
\[
  \fra[x]\defeq\fra[x,x], \qquad x\in\dom(\fra).
\]
Note that, for $x\in\dom(A)$ and $y\in\dom(\fra)$, we have $\fra[x,y]=\langle Ax,y\rangle$.
If $A$ is bounded from below, then this definition clearly coincides with the
definition in \cite[\S IV.1.5]{katopert}.
For more information on non-semi-bounded forms see, e.g.\ \cite{FHS00,GKMV}.

Let $\cL$ be a (not necessarily closed) subspace of $\dom(\fra)$.  We say
that $\cL$ is $\fra$-\emph{non-negative} if
\[
  \fra[x] \ge 0 \qquad \text{for every $x\in\cL$;}
\]
$\cL$ is called \emph{maximal} $\fra$-\emph{non-negative} if it cannot be
extended to a larger subspace with the same property.

Throughout the paper denote by $\cL_\Delta(A)$ the spectral subspace for $A$
corresponding to a Borel set $\Delta\subset\RR$, i.e.\ $\cL_\Delta(A)=\ran E(\Delta)$.

\bigskip

\noindent
\textbf{Assumptions (A1)--(A3).} \\
Let $T$ be an operator function defined on some interval $\Delta\subset\RR$
whose values are operators in a Hilbert space $\cH$.
We assume that the following conditions are satisfied:
\newcounter{counter_assump}
\begin{list}{{\rm\textbf{(A\arabic{counter_assump})}}}%
{\usecounter{counter_assump}
\setlength{\itemsep}{0.5ex}\setlength{\topsep}{1.1ex}
\setlength{\leftmargin}{7ex}\setlength{\labelwidth}{7ex}}
\item
$T(\lambda)$ is self-adjoint for every $\lambda\in\Delta$
with corresponding quadratic form $\frt(\lambda)$;
\item
$\dom(\frt(\lambda))=\dom\bigl(|T(\lambda)|^{1/2}\bigr)$ is independent of $\lambda$
and denoted by $\cD$;
\item
for each $x\in\cD\setminus\{0\}$, the function $\lambda\mapsto\frt(\lambda)[x]$
is continuous and decreasing at value $0$, i.e.\ if $\frt(\lambda_0)[x]=0$ for
some $\lambda_0\in\Delta$, then
\begin{alignat*}{2}
  \frt(\lambda)[x] &> 0 \qquad & &\text{for }\lambda\in\Delta\text{ such that }\lambda<\lambda_0, \\[0.5ex]
  \frt(\lambda)[x] &< 0 \qquad & &\text{for }\lambda\in\Delta\text{ such that }\lambda>\lambda_0.
\end{alignat*}
\end{list}

\noindent
Occasionally --- in particular, when the essential spectrum is involved --- we need
the following condition, which is named after A.~Virozub and V.~Matsaev (see \cite{VM74}
and also \cite{LMM06,LLMT08}):
\begin{list}{{\rm\textbf{\VMbf}}}%
{
\setlength{\itemsep}{0.5ex}\setlength{\topsep}{1.1ex}
\setlength{\leftmargin}{10ex}\setlength{\labelwidth}{10ex}}
\item
for every $u\in\cD$, the function $\frt(\cdot)[u]$ is differentiable on $\Delta$ and,
for every compact subinterval $I$ of $\Delta$, there exist $\eps,\delta>0$
such that, for all $x\in\cD$ with $\|x\|=1$ and all $\lambda\in I$,
\begin{equation}\label{epsdeltaimpl}
  \bigl|\frt(\lambda)[x]\bigr| \le \eps \quad\Longrightarrow\quad
  \frt'(\lambda)[x] \le -\delta.
\end{equation}
\end{list}

\noindent
Obviously, for fixed $\lambda\in I$, this condition (i.e.\ \eqref{epsdeltaimpl}
for $x\in\cD$ with $\|x\|=1$) is equivalent to the condition that
\begin{equation}\label{VMepsdelta}
  \bigl|\frt(\lambda)[x]\bigr| \le \eps\|x\|^2 \quad\Longrightarrow\quad
  \frt'(\lambda)[x] \le -\delta\|x\|^2
\end{equation}
for all $x\in\cD$.

In \cite{VM74,LMM06,LLMT08} the Virozub--Matsaev condition was studied
with $\frt'(\lambda)\ge\delta$ instead of $\frt'(\lambda)\le-\delta$.
Moreover, the definition was slightly different
but equivalent to ours (apart from the different sign) for the functions considered
in \cite{LLMT08}, which were assumed to have bounded operators as values and to
be continuously differentiable in the operator norm, cf.\ \cite[Lemma~3.6]{LLMT08}.

Next we define the notion of a generalised Rayleigh functional.
First note that, by Assumption (A3), the function $\lambda\mapsto\frt(\lambda)[x]$
has at most one zero for a given $x\in\cD\setminus\{0\}$.  If it has a zero,
we define a generalised Rayleigh functional $p(x)$ to be equal to this zero;
otherwise, we assign a value outside $\Delta$.  More precisely, we define
a generalised Rayleigh functional as follows.

\begin{definition}\label{def:rayleigh_Mg+}
Let $T$ be an operator function defined on $\Delta$ that satisfies Assumptions (A1)--(A3)
and let $\frt(\lambda)$ be the corresponding forms.
\begin{myenum}
\item
A functional $p:\cD\setminus\{0\}\to\RR\cup\{\pm\infty\}$ is called
a \emph{generalised Rayleigh functional for $T$ on} $\Delta$ if, for all $x\in\cD\setminus\{0\}$,
\begin{alignat*}{2}
  p(x) &= \lambda_0 & &\text{if } \frt(\lambda_0)[x]=0, \\[1ex]
  p(x) &< \lambda \;\;\text{for all }\lambda\in\Delta \qquad &
  &\text{if }\frt(\mu)[x]<0 \;\;\text{for all } \mu\in\Delta, \\[1ex]
  p(x) &> \lambda \;\;\text{for all }\lambda\in\Delta \qquad &
  &\text{if }\frt(\mu)[x]>0 \;\;\text{for all } \mu\in\Delta.
\end{alignat*}
\item
For $\gamma\in \Delta$ set
\[
  \M_\gamma^+ \defeq \bigl\{\cM: \cM\text{ is a maximal
  $\frt(\gamma)$-non-negative subspace of }\cD\bigr\}.
\]
\end{myenum}
\end{definition}

\noindent
In \cite{BEL00}, \cite{EL02} and \cite{EL04} generalised Rayleigh functionals were defined
such that $p(x)=-\infty$ and $p(x)=+\infty$ in the second and third case in (i) above.
This does not change results, but our definition gives more flexibility in applications;
cf.\ also \cite[\S 3]{JLT}.
Note that the choice with $\pm\infty$ is also allowed in our definition.
Note that, for all $\lambda\in\Delta$ and $x\in\cD\setminus\{0\}$,
\begin{equation} \label{T104}
  p(x) \lesseqqgtr \lambda \quad\Longleftrightarrow\quad
  \frt(\lambda)[x] \lesseqqgtr 0.
\end{equation}
Moreover, if $\lambda_0$ is an eigenvalue of $T$ with eigenvector $x_0$,
i.e.\ $T(\lambda_0)x_0=0$, then $p(x_0)=\lambda_0$.

The next two theorems are the main results of this section. The first one can be
used to show that some point is in the resolvent set of an operator function.
The second one, which is a generalisation of \cite[Theorem~2.4]{EL02}, gives
triple variational inequalities for eigenvalues and the
bottom of the essential spectrum of an operator function.

\begin{theorem} \label{th:resolv}
Assume that $T$ satisfies {\rm(A1)--(A3)} and \VM.
Let $\mu_1,\mu_2\in\Delta$ with $\mu_1<\mu_2$.
If there exist $\cM\in\M_{\mu_1}^+$ and $a>0$ such that
\begin{equation}\label{cond_resolv}
  \frt(\mu_2)[x] \ge a\|x\|^2 \qquad \text{for all } x\in\cM,
\end{equation}
then $\mu_2\in\rho(T)$.
\end{theorem}

\begin{theorem} \label{th:varinequ}
Let $\Delta\subset\RR$ be an interval with right endpoint $\beta\in\RR\cup\{+\infty\}$
and let $T$ be an operator function defined on $\Delta$ which satisfies Assumptions {\rm(A1)--(A3)}.
Moreover, let $p$ be a generalised Rayleigh functional for $T$ on $\Delta$,
let $\gamma\in\rho(T)$ with $\gamma<\beta$, and set
\begin{equation}\label{deflae}
  \lae \defeq \begin{cases}
    \inf\bigl(\sess(T)\cap(\gamma,\beta)\bigr)
    & \text{if } \sess(T)\cap(\gamma,\beta)\ne\varnothing, \\[1ex]
    \beta & otherwise.
  \end{cases}
\end{equation}
Let $(\lambda_j)_{j=1}^N$, $N\in\NN_0\cup\{\infty\}$, be a finite or infinite
sequence of eigenvalues of\, $T$ in the interval $(\gamma,\lae)$ in non-decreasing order
such that, for each set of $k$ coinciding eigenvalues, say
$\lambda_i=\lambda_{i+1}=\ldots=\lambda_{i+k-1}$, one has $\dim\ker T(\lambda_i)\ge k$.
Then
\begin{equation} \label{varinequ}
  \sup_{\cM\in\M_\gamma^+} \sup_{\substack{\cL\subset\cM \\[0.2ex] \dim\cL=n-1}}
  \inf_{\substack{x\in\cM\setminus\{0\} \\[0.2ex] x\perp\cL}}\; p(x) \le \lambda_n,
  \qquad n\in\NN,\,n\le N.
\end{equation}
Moreover, if $T$ satisfies the condition \VM{} and
$\sess(T)\cap(\gamma,\beta)\ne\varnothing$, then
\begin{equation} \label{varinequ_e}
  \sup_{\cM\in\M_\gamma^+} \sup_{\substack{\cL\subset\cM \\[0.2ex] \dim\cL=n-1}}
  \inf_{\substack{x\in\cM\setminus\{0\} \\[0.2ex] x\perp\cL}}\; p(x) \le \lae,
  \qquad n\in\NN.
\end{equation}
\end{theorem}

\medskip

\begin{remark}\label{re:mainths}
\rule{0ex}{1ex}
\begin{myenum}
\item
The statement in Theorem~\ref{th:resolv} is false without the assumption \VM{}
as can be seen from the following example.  Let $A$ be a self-adjoint operator
in a Hilbert space $\cH$ with spectrum $\sigma(A)=[0,1]$ but having $0$ not as
an eigenvalue. The operator function $T(\lambda)=-\lambda^2 I-A$ satisfies
Assumptions {\rm(A1)--(A3)} since $\frt(\lambda)[x]<0$ for all $x\in\cH\setminus\{0\}$
and $\lambda\in\RR$.
If we choose $\mu_1=-1$ and $\mu_2=0$, then $\M_{-1}^+=\{\{0\}\}$ and therefore
relation \eqref{cond_resolv} is satisfied.  However, $0\in\sigma(T)$.
\item
Note that the variations on the left-hand sides of \eqref{varinequ}
and \eqref{varinequ_e} are over non-empty sets for those $n$ considered there,
i.e.\ there exists an $\cM\in\M_\gamma^+$ which is at least $n$-dimensional;
see the beginning of the proof of Theorem~\ref{th:varinequ}.
\item
Under our assumptions one obtains in general only an inequality and not
equality as the following example shows.  Consider the operator function
$T(\lambda)=\diag(T_1(\lambda),T_2(\lambda),\dots)$, $\lambda\in\Delta=\RR$,
in the space $\cH=\ell^2$ where the piece-wise linear functions $T_k$
are defined as
\[
  T_k(\lambda) = \begin{cases}
    1, & \lambda\le0, \\[0.5ex]
    1-k\lambda, & 0<\lambda<\frac{2}{k}\,, \\[0.5ex]
    -1, & \lambda\ge\frac{2}{k}\,.
  \end{cases}
\]
The spectrum of $T$ consists only of eigenvalues:
\[
  \sigma(T)=\sigma_{\rm p}(T)=\biggl\{\frac{1}{k}:k\in\NN\biggr\},
\]
but the variations on the left-hand side of \eqref{varinequ} are equal to $0$
for all $n\in\NN$ if one chooses, e.g.\ $\gamma=0$.
\item
Note that in the last statement of the theorem the condition \VM{} is necessary
as can be seen from the example given in \cite[Remark~2.10]{EL04}.
\end{myenum}
\end{remark}

\noindent
Before we prove the theorems, we need a couple of lemmas.

\begin{lemma} \label{le:00}
Let $A$ be a self-adjoint operator, $\fra$ the corresponding form, and assume
that $0\in\rho(A)$.  Let $\cM$ be a maximal $\fra$-non-negative subspace
of $\dom(\fra)=\dom(|A|^{1/2})$, $\cM'$ an $\fra$-non-negative subspace
of $\dom(\fra)$ and $\cL\subset\cM\cap\cM'$.  Then
\[
  \dim(\cM/\cL) \ge \dim(\cM'/\cL).
\]
\end{lemma}

\begin{proof}
Since $0\in\rho(A)$, $\cK\defeq\dom(\fra)$ is a Krein space with inner product
$\fra[\cdot\,,\cdot]$, i.e.\ it is a direct and orthogonal sum of the Hilbert
space $\cK_+=\cL_{(0,\infty)}(A)\cap\dom(\fra)$ and the anti-Hilbert space
$\cK_-=\cL_{(-\infty,0)}(A)\cap\dom(\fra)$.
According to \cite[Proposition~I.1.1]{L82} and its first corollary,
$\cM$ and $\cM'$ have angular operator representations, i.e., with respect to
the decomposition $\cK=\cK_+\,\dot{+}\,\cK_-$, they can be written as
\begin{equation}\label{angular}
  \cM = \biggl\{\binom{x}{C_{\cM}x}: x\in\cK_+\biggr\}, \qquad
  \cM' = \biggl\{\binom{x}{C_{\cM'}x}: x\in\dom(C_{\cM'})\biggr\},
\end{equation}
where $C_\cM$ and $C_{\cM'}$ are bounded operators from $\cK_+$ to $\cK_-$ with
$\dom(C_\cM)=\cK_+$ and $\dom(C_{\cM'})\subset\cK_+$.  This implies that $\cM$ is isomorphic
to $\cK_+$ and $\cM'$ is isomorphic to a subspace of $\cK_+$.
From this the claim is immediate.
\end{proof}

\begin{lemma} \label{le:01}
Let $\fra$ be a quadratic form with domain $\dom(\fra)$ corresponding to a
self-adjoint operator $A$.  Let $u\in\dom(\fra)$, $v\in\dom(A)$ and
$a,b,c>0$ such that
\[
  \fra[u] \ge a\|u\|^2, \qquad \|Av\| \le b\|v\|
\]
and
\[
  ac>b(a+b+3c).
\]
If\, $u\ne0$ or $v\ne0$, then
\[
  \fra[u+v] + c\|u+v\|^2 > 0.
\]
\end{lemma}

\begin{proof}
We can estimate
\begin{align}
  &\fra[u+v] + c\|u+v\|^2 \notag\\[1ex]
  &= \fra[u] + 2\Re\langle Av,u\rangle + \langle Av,v\rangle
  + c\|u\|^2 + 2c\Re\langle v,u\rangle + c\|v\|^2 \notag\\[1ex]
  &\ge \fra[u] - 2\|Av\|\,\|u\| - \|Av\|\,\|v\|
  + c\|u\|^2 - 2c\|u\|\,\|v\| + c\|v\|^2 \notag\\[1ex]
  &\ge a\|u\|^2 - 2b\|v\|\,\|u\| - b\|v\|^2
  + c\|u\|^2 - 2c\|u\|\,\|v\| + c\|v\|^2 \notag\\[1ex]
  &= (a+c)\|u\|^2 - 2(b+c)\|u\|\,\|v\| + (c-b)\|v\|^2. \label{T101}
\end{align}
Since $a+c>0$, the quadratic form in $\|u\|$ and $\|v\|$ is positive definite
if and only if
\[
  (a+c)(c-b)-(b+c)^2>0,
\]
which is equivalent to
\[
  ac > b(a+b+3c).
\]
As this inequality is true by assumption, the expression in \eqref{T101}
is positive unless both $\|u\|$ and $\|v\|$ are zero.
\end{proof}

\medskip

In the next lemmas $T$ is an operator function defined on an interval $\Delta$.

\begin{lemma} \label{le:02}
Assume that $T$ satisfies {\rm(A1)--(A3)} and \VM.
Let $\mu_1,\mu_2\in\Delta$ with $\mu_1<\mu_2$ and let
$\eps$ and $\delta$ be such that \eqref{VMepsdelta} is valid for all
$\lambda\in[\mu_1,\mu_2]$ and $x\in\cD$.
Then
\[
  \frt(\mu_2)[x] \ge -\eps\|x\|^2 \quad\Longrightarrow\quad
  \frt(\mu_1)[x] \ge \min\bigl\{\eps\|x\|^2,\,
  \frt(\mu_2)[x]+\delta(\mu_2-\mu_1)\|x\|^2\bigr\}
\]
for all $x\in\cD$.
\end{lemma}

\begin{proof}
Without loss of generality we may assume that $\|x\|=1$.
If $\frt(\lambda_0)[x]\ge\eps$ for some $\lambda_0\in[\mu_1,\mu_2]$, then, clearly,
$\frt(\lambda)[x]\ge\eps$ for all $\lambda\in[\mu_1,\lambda_0]$;
if $\frt(\lambda_0)[x]\le-\eps$ for some $\lambda_0\in[\mu_1,\mu_2]$,
then $\frt(\lambda)[x]\le-\eps$ for all $\lambda\in[\lambda_0,\mu_2]$.
Now, if $\frt(\mu_1)[x]\ge\eps$, then there is nothing to prove.
Otherwise, $\frt(\lambda)[x]\in[-\eps,\eps]$ for all $\lambda\in[\mu_1,\mu_2]$ and
therefore $\frt'(\lambda)[x]\le-\delta$ for such $\lambda$.  Hence
\[
  \frt(\mu_1)[x] = \frt(\mu_2)[x]-\int_{\mu_1}^{\mu_2} \frt'(\lambda)[x]\,\rd\lambda
  \ge \frt(\mu_2)[x]+\delta(\mu_2-\mu_1),
\]
which shows the assertion.
\end{proof}

\begin{lemma} \label{le:03}
Assume that $T$ satisfies {\rm(A1)--(A3)} and \VM.
Let $\mu_1,\mu_2\in\Delta$ with $\mu_1<\mu_2$ and
let $\eps$ and $\delta$ be such that \eqref{VMepsdelta} is valid for all
$\lambda\in[\mu_1,\mu_2]$ and $x\in\cD$.
Moreover, let $a,b>0$, set $c\defeq\min\{\eps,\,\delta(\mu_2-\mu_1)\}$
and suppose that
\[
  ac > b(a+b+3c).
\]
If $u\in\cD$, $v\in\dom(T(\mu_2))$ are such that
\[
  \frt(\mu_2)[u] \ge a\|u\|^2, \qquad \|T(\mu_2)v\| \le b\|v\|
\]
and $u\ne0$ or $v\ne0$, then
\[
  \frt(\mu_1)[u+v] > 0.
\]
\end{lemma}

\begin{proof}
It follows from Lemma~\ref{le:01} applied to $\fra=\frt(\mu_2)$ that
\begin{equation} \label{T102}
  \frt(\mu_2)[u+v]+c\|u+v\|^2 > 0.
\end{equation}
Since $c\le\eps$, we have $\frt(\mu_2)[u+v] \ge -\eps\|u+v\|^2$.
Now Lemma~\ref{le:02} implies that
\[
  \frt(\mu_1)[u+v] \ge \min\bigl\{\eps\|u+v\|^2,\,
  \frt(\mu_2)[u+v]+\delta(\mu_2-\mu_1)\|u+v\|^2\bigr\}.
\]
The first expression in the minimum is positive because $u+v\ne0$ by \eqref{T102}.
The second expression in the minimum is also positive:
\[
  \frt(\mu_2)[u+v]+\delta(\mu_2-\mu_1)\|u+v\|^2
  > -c\|u+v\|^2 + \delta(\mu_2-\mu_1)\|u+v\|^2 \ge 0
\]
by the definition of $c$, which implies the assertion.
\end{proof}

\begin{lemma} \label{le:04}
Assume that $T$ satisfies {\rm(A1)--(A3)} and \VM.
Let $a>0$, $\mu_1,\mu_2\in\Delta$ with $\mu_1<\mu_2$, and let $\cM$ be a
subspace of $\cD$.
Moreover, suppose that $\mu_2\in\sigma(T)$ and that
\begin{equation} \label{T103}
  \frt(\mu_2)[x] \ge a\|x\|^2 \qquad \text{for all } x\in\cM.
\end{equation}
Then there exists a subspace $\cM'$ such that $\cM \subsetneq \cM' \subset \cD$ and
\[
  \frt(\mu_1)[x] > 0 \qquad\text{for all } x\in\cM',\,x\ne0.
\]
\end{lemma}

\begin{proof}
Let $\eps$ and $\delta$ be such that \eqref{VMepsdelta} is valid for all
$\lambda\in[\mu_1,\mu_2]$ and $x\in\cD$.
Set $c\defeq\min\{\eps,\delta(\mu_2-\mu_1)\}$ and
choose a positive number $b$ such that $b<a$ and $ac>b(a+b+3c)$.
Since $\mu_2\in\sigma(T)$, there exists a $v_0\in\dom(T(\mu_2))\subset\cD$ such
that $\|v_0\|=1$ and $\|T(\mu_2)v_0\| \le b$.  Set $\cM' \defeq \cM+\spn\{v_0\}$.
The space $\cM'$ is strictly larger than $\cM$ because $b<a$ and \eqref{T103}
is satisfied.  Now let $u\in\cM$ and $v\in\spn\{v_0\}$.  Then
$\frt(\mu_1)[u+v] > 0$ if $u+v\ne0$ by Lemma~\ref{le:03}.
\end{proof}

\medskip

Theorem~\ref{th:resolv} is now an immediate consequence of the previous lemma.

\begin{proof}[Proof of Theorem~\ref{th:resolv}]
If $\mu_2\in\sigma(T)$, then, by Lemma~\ref{le:04}, there exists $\cM'\subset\cD$,
$\cM\subsetneq\cM'$ such that $\frt(\mu_1)[x]\ge0$ for all $x\in\cM'$,
which contradicts the maximality of $\cM$ as a $\frt(\mu_1)$-non-negative
subspace.
\end{proof}

Before we prove Theorem~\ref{th:varinequ} we need two more lemmas.

\begin{lemma} \label{le:05}
Assume that $T$ satisfies {\rm(A1)--(A3)}.
Let $\lambda_1,\dots,\lambda_m$ be eigenvalues of\, $T$ with eigenvectors
$u_1,\dots,u_m$ and let $\mu,\nu\in\Delta$ such that
$\mu\le\lambda_1\le\dots\le\lambda_m\le\nu$.
Moreover, let $y\in\cD$ and $c_1,\dots,c_m\in\CC$.
\begin{myenum}
\item
If\, $\frt(\nu)[y]\ge0$, then
\[
  \frt(\mu)[y+c_1u_1+\ldots+c_m u_m] \ge 0.
\]
\item
If\, $\frt(\mu)[y]\le0$, then
\[
  \frt(\nu)[y+c_1u_1+\ldots+c_m u_m] \le 0.
\]
\end{myenum}
\end{lemma}

\begin{proof}
We prove only the assertion in (i); the statement in (ii) is proved analogously.

Since $\frt(\nu)[y]\ge0$ and $T$ satisfies Assumption (A3), we have
$\frt(\lambda_m)[y]\ge0$.
Using the fact that $\lambda_m$ is an eigenvalue of $T$ with eigenvector $u_m$,
i.e.\ that $T(\lambda_m)u_m=0$, we obtain
\begin{align*}
  \frt(\lambda_m)[y+c_m u_m]
  &= \frt(\lambda_m)[y] + 2\Re\bigl\langle c_m T(\lambda_m)u_m,y\bigr\rangle
  + |c_m|^2\bigl\langle T(\lambda_m)u_m,u_m\bigr\rangle \\[1ex]
  &= \frt(\lambda_m)[y] \ge 0
\end{align*}
and hence, again by Assumption (A3), $\frt(\lambda_{m-1})[y+c_m u_m] \ge 0$.
Repeating this argument we get
\[
  \frt(\lambda_1)[y+c_1u_1+\dots+c_m u_m] \ge0.
\]
Finally, we can once more use Assumption (A3) to prove the claim.
\end{proof}

\begin{lemma} \label{le:06}
Assume that $T$ satisfies {\rm(A1)--(A3)}.
Let $\lambda_1\le\lambda_2\le\dots\le\lambda_m$ be eigenvalues of\, $T$
such that, for each set of $k$ coinciding eigenvalues, say
$\lambda_i=\lambda_{i+1}=\ldots=\lambda_{i+k-1}$, one has $\dim\ker T(\lambda_i)\ge k$.
Then there exist linearly independent vectors $u_1,\dots,u_m$ such that $u_j$
is an eigenvector of $T$ corresponding to $\lambda_j$, $j=1,\dots,m$.
\end{lemma}

\begin{proof}
For every $\lambda_j$ choose an eigenvector $u_j$ such that for coinciding
eigenvalues the eigenvectors are linearly independent.
Assume that there exist numbers $\alpha_1,\dots,\alpha_m\in\CC$,
not all equal to $0$, such that
\[
  \alpha_1 u_1+\ldots+\alpha_m u_m = 0.
\]
Let $\alpha_n$ be the last non-zero coefficient, i.e.\ $\alpha_n\ne0$ and
\[
  \alpha_1 u_1+\ldots+\alpha_n u_n = 0.
\]
Because the $u_j$ are chosen to be linearly independent for coinciding
eigenvalues, we have $\lambda_1<\lambda_n$.  Let $k$ be such that
\[
  \lambda_k < \lambda_{k+1} = \ldots = \lambda_n.
\]
Since $u_{k+1},\dots,u_n$ are linearly independent and $\alpha_n\ne0$,
it follows that
\[
  \alpha_1 u_1 + \ldots + \alpha_k u_k
  = -(\alpha_{k+1}u_{k+1}+\ldots+\alpha_n u_n) \ne 0.
\]
From Lemma~\ref{le:05}\,(ii) with $y=0$ we obtain that
\begin{equation} \label{T108}
  \frt(\lambda_k)[\alpha_1 u_1 + \ldots + \alpha_k u_k] \le 0.
\end{equation}
The fact that $\alpha_{k+1}u_{k+1}+\ldots+\alpha_n u_n$ is an
eigenvector to the eigenvalue $\lambda_n$ implies that
$\frt(\lambda_n)[\alpha_{k+1}u_{k+1}+\ldots+\alpha_n u_n]=0$.  Hence
\begin{align*}
  0 &= \frt(\lambda_n)[\alpha_{k+1}u_{k+1}+\ldots+\alpha_n u_n]
  = \frt(\lambda_n)[\alpha_1 u_1 + \ldots + \alpha_k u_k] \\[1ex]
  &<\frt(\lambda_k)[\alpha_1 u_1 + \ldots + \alpha_k u_k]
\end{align*}
by (A3), which is a contradiction to \eqref{T108}.
\end{proof}

Note that, without assumption (A3), the statement of the previous lemma is false
in general; see, e.g.\ the example in \cite[Remark~7.7]{LMM06}.

\medskip

Now we can turn to the proof of Theorem~\ref{th:varinequ}.

\begin{proof}[Proof of Theorem~\ref{th:varinequ}]
First we show Remark~\ref{re:mainths}\,(ii).  Let $n\in\NN$ and assume that $T$
has at least $n$ eigenvalues in $(\gamma,\lae)$ counted with multiplicities.
It follows from Lemma~\ref{le:06} that
there exist linearly independent eigenvectors $u_1,\dots,u_n$ of $T$ corresponding
to $\lambda_1,\dots,\lambda_n$.  By Lemma~\ref{le:05}\,(i) the space
$\spn\{u_1,\dots,u_n\}$ is $\frt(\gamma)$-non-negative and can be extended to
a maximal $\frt(\gamma)$-non-negative subspace by Zorn's lemma, which shows
the statement concerning \eqref{varinequ}.
For the analogous statement for \eqref{varinequ_e} let
$\mu_2\in\sess(T)\cap(\gamma,\beta)$ and $a,b,c$ as in Lemma~\ref{le:03}
where $\eps,\delta$ are such that \eqref{VMepsdelta} is valid on $[\gamma,\mu_2]$.
If $n\in\NN$, then there exists an $n$-dimensional subspace $\cM'$ of $\dom(T(\mu_2))$
such that $\|T(\mu_2)v\|\le b\|v\|$ for $v\in\cM'$.  It follows from Lemma~\ref{le:03}
with $u=0$ that the space $\cM'$ is $\frt(\gamma)$-non-negative.
Again we can extend this space to a maximal $\frt(\gamma)$-non-negative subspace.

Suppose that the inequality in \eqref{varinequ} is false for some $n$.
Then there exist an $\cM\in\M_\gamma^+$ and a subspace $\cL\subset\cM$
with $\dim\cL=n-1$ such that
\[
  \inf_{\substack{x\in\cM\setminus\{0\} \\[0.2ex] x\perp\cL}} p(x) > \lambda_n.
\]
Hence
\begin{equation} \label{T107}
  \frt(\lambda_n)[y] > 0, \qquad y\in\cM\ominus\cL,\,y\ne0.
\end{equation}
By Lemma~\ref{le:06} there exist linearly independent eigenvectors $u_1,\dots,u_n$
corresponding to the eigenvalues $\lambda_1,\dots,\lambda_n$.
According to Lemma~\ref{le:05}\,(i) we have
\[
  \frt(\gamma)[y+c_1u_1+\dots+c_n u_n] \ge 0
\]
for all $y\in\cM\ominus\cL$ and $c_1,\dots,c_n\in\CC$.  This implies that
\begin{equation} \label{T106}
  \cM' \defeq (\cM\ominus\cL)+\spn\{u_1,\dots,u_n\}
\end{equation}
is a $\frt(\gamma)$-non-negative subspace of $\cD$.
The sum in \eqref{T106} is direct because of \eqref{T107} and
\[
  \frt(\lambda_n)[x] \le 0, \qquad x\in\spn\{u_1,\dots,u_n\},
\]
which is true by Lemma~\ref{le:05}\,(ii).  Hence
\[
  \dim \bigl(\cM'/(\cM\ominus\cL)\bigr) = n, \qquad \dim \bigl(\cM/(\cM\ominus\cL)\bigr) = n-1.
\]
Lemma~\ref{le:00} shows that this contradicts the maximality of $\cM$.

For the second part assume that the inequality in \eqref{varinequ_e} is false
for some $n\in\NN$.  There there exist an $\cM\in\M_\gamma^+$ and a subspace
$\cL\subset\cM$ with $\dim\cL = n-1$ such that
\begin{equation} \label{T110}
  \mu_2 \defeq \inf_{\substack{x\in\cM\setminus\{0\} \\[0.2ex] x\perp\cL}} p(x)
  > \lae.
\end{equation}
According to the definition of $\lae$ there exists a number $\mu_1\in\sess(T)$
so that $\lae\le\mu_1<\mu_2$.
It follows from \eqref{T110} that $\frt(\mu_2)[x]\ge0$ for $x\in\cM\ominus\cL$
and hence from Lemma~\ref{le:02} that
\begin{equation} \label{T111}
  \frt(\mu_1)[x] \ge a\|x\|^2, \qquad x\in\cM\ominus\cL,
\end{equation}
where $a\defeq\min\{\eps,\delta(\mu_2-\mu_1)\}$
and $\eps,\delta$ are such that \eqref{VMepsdelta} is valid for all
$\lambda\in[\gamma,\mu_2]$ and $x\in\cD$.
Set $c\defeq\min\{\eps,\delta(\mu_1-\gamma)\}$ and choose $b>0$ such that
$b<a$ and $ac>b(a+b+3c)$.
Since $\mu_1\in\sess(T)$, i.e.\ $0\in\sess(T(\mu_1))$, there exists
an $n$-dimensional subspace $\cV$ of $\dom(T(\mu_1))\subset\cD$ such that
\begin{equation} \label{T112}
  \|T(\mu_1)v\| \le b\|v\| \qquad \text{for all } v\in\cV.
\end{equation}
Set $\cM'\defeq (\cM\ominus\cL)\,\dot{+}\,\cV$; the sum is direct because of
\eqref{T111}, \eqref{T112} and the inequality $b<a$.
It follows from \eqref{T111}, \eqref{T112} and Lemma~\ref{le:03} that
\begin{equation} \label{T113}
  \frt(\gamma)[y] \ge 0 \qquad \text{for all } y \in \cM',
\end{equation}
i.e.\ $\cM'$ is $\frt(\gamma)$-non-negative.  Since
\[
  \dim\bigl(\cM'/(\cM\ominus\cL)\bigr) = n, \qquad
  \dim\bigl(\cM/(\cM\ominus\cL)\bigr) = n-1,
\]
this is a contradiction to the maximality of $\cM$ according to Lemma~\ref{le:00}.
\end{proof}

\section{Self-adjoint operators}\label{sec:selfadjoint_op}

Let $A$ be a self-adjoint operator and $\fra$ the corresponding quadratic form
with domain $\cD \defeq \dom(\fra) = \dom\bigl(|A|^{1/2}\bigr)$.
We introduce the operator function $T(\lambda) = A-\lambda I$ and the associated form
$\frt(\lambda)[x,y] = \fra[x,y]-\lambda\langle x,y\rangle$, where $\lambda\in\RR$
and $x,y\in\cD$.
Note that $T$ satisfies Assumptions (A1)--(A3) and the condition \VM{}
on any interval since $\frt'(\lambda)[x] = -\|x\|^2$.
As in Definition~\ref{def:rayleigh_Mg+}\,(ii) set
\[
  \M_\gamma^+ \defeq \bigl\{\cM\colon \cM\text{ is a maximal
  $(\fra-\gamma)$-non-negative subspace of }\cD\bigr\}
\]
for $\gamma\in\RR$.

In the following theorem eigenvalues in a gap of the essential spectrum are
characterised by a triple variational principle.
This result is a generalisation of \cite[Theorem~3.1]{EL02}
to unbounded operators.
For other types of variational principles for eigenvalues of self-adjoint operators
in gaps of the essential spectrum see, e.g.\ \cite{DolEsSere06,GS99,KLT04,LLT02},
where a given decomposition of the space is used.
Note that Theorem~\ref{th:selfadjointops} is not a corollary of Theorem~\ref{th:var_equ}
below since there we assume that the values of the operator function are
operators that are bounded from below,
which is not assumed in Theorem~\ref{th:selfadjointops}.

\begin{theorem} \label{th:selfadjointops}
Let $\gamma\in\rho(A)\cap\RR$ and set
\[
  \lae\defeq\inf\bigl(\sess(A)\cap(\gamma,\infty)\bigr).
\]
Moreover, let $(\lambda_j)_{j=1}^N$, $N\in\NN_0\cup\{\infty\}$, be the finite or
infinite sequence of eigenvalues in $(\gamma,\lae)$ in non-decreasing order and
counted according to their multiplicities: $\lambda_1\le\lambda_2\le\cdots$. Then
\begin{equation} \label{varequsao}
  \lambda_n = \sup_{\cM\in\M_\gamma^+} \;\; \sup_{\substack{\cL\subset\cM \\[0.2ex] \dim\cL=n-1}} \;\;
  \inf_{\substack{x\in\cM\setminus\{0\} \\[0.2ex] x\perp\cL}} \;\; \frac{\fra[x]}{\|x\|^2}\,,
  \qquad n\in\NN,\;n\le N.
\end{equation}
Moreover, if $N$ is finite and $\sess(A)\cap(\gamma,\infty)\ne\varnothing$, then
\begin{equation} \label{varequsaoess}
  \min\bigl(\sess(A)\cap(\gamma,\infty)\bigr) = \sup_{\cM\in\M_\gamma^+} \;\;
  \sup_{\substack{\cL\subset\cM \\[0.2ex] \dim\cL=n-1}} \;\;
  \inf_{\substack{x\in\cM\setminus\{0\} \\[0.2ex] x\perp\cL}} \;\; \frac{\fra[x]}{\|x\|^2}\,,
  \qquad n > N.
\end{equation}
\end{theorem}

\begin{proof}
The inequalities `$\le$' in \eqref{varequsao} and \eqref{varequsaoess} follow from
Theorem~\ref{th:varinequ} since
\[
  p(x)=\frac{\fra[x]}{\|x\|^2}
\]
is a generalised Rayleigh functional for the operator function $T$ on $(\gamma,\infty)$.
To show the reverse inequalities, set $\cM\defeq\cL_{(\gamma,\infty)}(A)\cap\cD$
where $\cL_{(\gamma,\infty)}(A)$ denotes the spectral subspace for $A$ corresponding
to the interval $(\gamma,\infty)$.
Clearly, $\cM$ is maximal $\frt(\gamma)$-non-negative because
\[
  \cD = \bigl(\cL_{(-\infty,\gamma)}(A)\cap\cD\bigr)\,\dot{+}\,
  \bigl(\cL_{(\gamma,\infty)}(A)\cap\cD\bigr).
\]
The operator $A|_{\dom(A)\cap\cL_{(\gamma,\infty)}(A)}$ is self-adjoint in
$\cH'\defeq\cL_{(\gamma,\infty)}(A)$ and bounded from below.
A standard variational principle yields that
\begin{alignat*}{2}
  \sup_{\substack{\cL\subset\cM \\[0.2ex] \dim\cL=n-1}}\;\;
  \inf_{\substack{x\in\cM\setminus\{0\} \\[0.2ex] x\perp\cL}}\; \frac{\fra[x]}{\|x\|^2} &= \lambda_n,
  &\qquad & n\in\NN,\;n\le N, \\[1ex]
  \sup_{\substack{\cL\subset\cM \\[0.2ex] \dim\cL=n-1}}\;\;
  \inf_{\substack{x\in\cM\setminus\{0\} \\[0.2ex] x\perp\cL}}\; \frac{\fra[x]}{\|x\|^2} &= \lae,
  &\qquad & n > N,
\end{alignat*}
which shows the inequalities `$\ge$' in \eqref{varequsao} and \eqref{varequsaoess}.
\end{proof}

\begin{remark}\label{re:unifposM}
Let us denote by $\M_\gamma^{++}$ the set of $\cM\in\M_\gamma^+$ on which $\fra-\gamma$ is
uniformly positive, i.e.\
\[
  \M_\gamma^{++} \defeq \bigl\{\cM\in\M_\gamma^+:
  \exists\,c>0 \;\;\text{such that}\;\; \fra[x]-\gamma\|x\|^2 \ge c\|x\|^2
  \;\;\text{for}\;\;x\in\cM\bigr\}.
\]
One can replace the first supremum in \eqref{varequsao} and \eqref{varequsaoess}
by $\sup_{\cM\in\M_\gamma^{++}}$ because $\M_\gamma^{++}\subset\M_\gamma^+$
and the maximising subspace that is used in the proof of Theorem~\ref{th:selfadjointops}
belongs to $\M_\gamma^{++}$.

Assume for the rest of this remark that $\gamma=0$, which is without loss of generality.
Let $\cM\in\M_0^{++}$ and let $C_{\cM}$ be as in \eqref{angular}.
Then $\|C_{\cM}\|<1$, and hence the form $\fra_{\cM} \defeq \fra|_{\cM}$
is a closed positive form in the Hilbert space $\ov{\cM}$.
Let $A_{\cM}$ be the representing operator of $\fra_{\cM}$ in the
sense of \cite[Theorem~VI.2.1]{katopert},
and let $\lambda_1(A_{\cM})\le\lambda_2(A_{\cM})\le\cdots$ be the eigenvalues
of $A_{\cM}$ below the essential spectrum; if there is only a finite number, say $N_{\cM}$,
of eigenvalues, then set $\lambda_n(A_{\cM})\defeq\min\sess(A_{\cM})$ for $n>N_{\cM}$.
The standard variational principle for semi-bounded operators yields
\[
  \lambda_n(A_{\cM}) = \sup_{\substack{\cL\subset\cM \\[0.2ex] \dim\cL=n-1}} \;\;
  \inf_{\substack{x\in\cM\setminus\{0\} \\[0.2ex] x\perp\cL}} \;\; \frac{\fra[x]}{\|x\|^2}\,,
  \qquad n\in\NN.
\]
Hence relation \eqref{varequsao} with $\cM_0^+$ replaced by $\cM_0^{++}$ turns into
\begin{equation}\label{supoverM}
  \lambda_n = \sup_{\cM\in\M_0^{++}} \lambda_n(A_{\cM}),
  \qquad n\in\NN,\;n\le N,
\end{equation}
and a similar relation holds for $\lae$ if $N$ is finite.
\end{remark}

\begin{example}
Consider the Dirac operator
\[
  D \defeq \sum_{j=1}^3 \upalpha_j\frac{1}{i}\partial_j + \upbeta + V
\]
in the space $L^2(\RR^3;\CC^4)$ where $\upalpha_j$ and $\upbeta$
are the $4\times4$ complex matrices
\[
  \upalpha_j = \begin{pmatrix} 0 & \upsigma_j \\[0.5ex] \upsigma_j & 0 \end{pmatrix}, \qquad
  \upbeta = \begin{pmatrix} I & 0 \\[0.5ex] 0 & -I \end{pmatrix}
\]
with $\upsigma_j$ being the Pauli matrices,
\[
  \upsigma_1 = \begin{pmatrix} 0 & 1 \\[0.5ex] 1 & 0 \end{pmatrix}, \qquad
  \upsigma_2 = \begin{pmatrix} 0 & -i \\[0.5ex] i & 0 \end{pmatrix}, \qquad
  \upsigma_3 = \begin{pmatrix} 1 & 0 \\[0.5ex] 0 & -1 \end{pmatrix},
\]
and where $V$ is the electrostatic potential.
Assume that $V$ is such that $D$ is a self-adjoint operator
with form domain $\cD=\dom(|D|^{1/2})$ 
and form $\frd$ as in \eqref{defform}.
If $\gamma\in\rho(D)$, then one can characterise eigenvalues of $D$
in $(\gamma,\min(\sess(D)\cap(\gamma,\infty)))$ with \eqref{varequsao}.

Let $D_0$ be the free Dirac operator, i.e.\ the operator from above with $V\equiv0$
and set $\cM_0\defeq\cL_{(0,\infty)}(D_0)\cap\cD$.
Assume that there exists $c>0$ such that
\begin{equation}\label{dirac_possubsp}
\begin{alignedat}{2}
  \frd[x] &\ge c\|x\|^2 \qquad &&\text{for}\;\; x\in\cM_0, \\[1ex]
  \frd[x] &< 0 &&\text{for}\;\; x\in\cM_0^\perp\cap\cD.
\end{alignedat}
\end{equation}
These conditions are satisfied, e.g.\ when
\begin{equation}\label{assumpV}
  -\frac{\mu}{|x|} \le V(x) \le 0
\end{equation}
with $0\le\mu\le2/\bigl(\frac{\pi}{2}+\frac{2}{\pi}\bigr)$; see \cite[Theorem~1]{tix98}.
If \eqref{dirac_possubsp} is satisfied, then $\cM_0\in\M_0^{++}$,
where $\M_0^{++}$ is as in Remark~\ref{re:unifposM}.
Then $\frd|_{\cM_0}$ defines a positive self-adjoint
operator $B$ in the Hilbert space $\cL_{(0,\infty)}(D_0)$,
which was called Brown--Ravenhall operator in the literature;
see, e.g.\ \cite{BR51,EPS96,tix98,GS99}.
Let $\lambda_n(B)$ be the eigenvalues of $B$ below its essential spectrum;
if $B$ has only a finite number, say $N_B$, of such eigenvalues,
then set $\lambda_n(B)\defeq\min\sess(B)$ for $n>N_B$.
Moreover, let $\lambda_n(D)$ be the eigenvalues of $D$ in the
interval $(0,\min(\sess(D)\cap(0,\infty)))$; again if there are only
finitely many such eigenvalues, say $N_D$, then
set $\lambda_n(D)=\min(\sess(D)\cap(0,\infty))$ for $n>N_D$.
Relation \eqref{supoverM} implies that
\[
  \lambda_n(B)\le\lambda_n(D), \qquad n\in\NN.
\]
This inequality was proved for $V$ satisfying \eqref{assumpV}
with $\mu<\sqrt{3}/2$ in \cite[Theorem~6]{GS99}.
The relation in \eqref{supoverM} also shows that the eigenvalues of the
Dirac operator are obtained by maximising the eigenvalues of operators
that are obtained in a similar way as $B$ but with arbitrary $\cM\in\M_0^{++}$.
%
\end{example}

The following theorem shows that for a non-negative perturbation of a self-adjoint
operator a spectral gap closes only from one side.

\begin{theorem}\label{th:nonneg_pert}
Let $A$ be a self-adjoint operator with corresponding quadratic form $\fra$
and $\alpha,\beta\in\RR$ such that $(\alpha,\beta)\subset\rho(A)$.
Moreover, let $\frb$ be a non-negative quadratic form with
$\dom(\frb)\supset\dom(\fra)$ such that $\fra+\frb$ with domain $\dom(\fra)$
is the quadratic form of a self-adjoint operator $C$, and assume that
\begin{equation}\label{rel_bdd_B_A}
  \frb[x] \le a\|x\|^2 + b\fra[x], \qquad x\in\dom(\fra),
\end{equation}
with some $a,b\ge0$.
If $\hat\alpha<\beta$ where
\begin{equation}\label{defalphahat}
  \hat\alpha\defeq \alpha+a+b\alpha,
\end{equation}
then $(\hat\alpha,\beta)\subset\rho(C)$.
\end{theorem}

\begin{proof}
%
Consider the operator function $T(\lambda)\defeq C-\lambda I$ with corresponding
forms $\frt(\lambda)=\fra+\frb-\lambda$ with domains $\cD=\dom(\fra)$ and the subspace
\[
  \cM\defeq\cL_{(\alpha,\infty)}(A)\cap\cD = \cL_{[\beta,\infty)}(A)\cap\cD.
\]
Let $\mu\in(\hat\alpha,\beta)$ and choose some $\mu_1\in(\hat\alpha,\mu)$.
For $x\in\cM$ and $\lambda\in(\alpha,\beta)$ we have
\begin{equation}\label{T120}
  \frt(\lambda)[x] = \fra[x]+\frb[x]-\lambda\|x\|^2
  \ge (\beta-\lambda)\|x\|^2.
\end{equation}
In particular, this shows that $\cM$ is a $\frt(\mu_1)$-non-negative subspace of $\cD$.
Assume that $\cM$ is not maximal $\frt(\mu_1)$-non-negative.
Then there exists a non-zero element $x_0$ in $\cL_{(-\infty,\alpha]}(A)\cap\cD$
such that $\frt(\mu_1)[x_0]\ge0$.  However,
\begin{align*}
  \frt(\mu_1)[x_0] &= \fra[x_0]+\frb[x_0]-\mu_1\|x_0\|^2 \\[1ex]
  &\le \fra[x_0]+a\|x_0\|^2+b\fra[x_0]-\mu_1\|x_0\|^2 \\[1ex]
  &\le \bigl((1+b)\alpha+a-\mu_1\bigr)\|x_0\|^2
  = (\hat\alpha-\mu_1)\|x_0\|^2 < 0,
\end{align*}
which is a contradiction.  Hence $\cM\in\M_{\mu_1}^+$.
Since $\beta>\mu$, it follows from \eqref{T120} with $\lambda$ replaced by $\mu$
and Theorem~\ref{th:resolv} that $\mu\in\rho(T)=\rho(C)$.
\end{proof}

If $A$ is bounded from below and $\frb$ is a non-negative form
with $\dom(\frb)\supset\dom(\fra)$, then $\fra+\frb$ with domain $\dom(\fra)$
is a closed form that is bounded from below (see, e.g.\ \cite[Theorem~VI.1.31]{katopert}).
Hence there exists a self-adjoint operator $C$ that represents the form $\fra+\frb$
by \cite[Theorem~VI.2.1]{katopert}.  Therefore Theorem~\ref{th:nonneg_pert}
can be applied if \eqref{rel_bdd_B_A} is satisfied.

If $B$ is a bounded non-negative operator, then $\hat\alpha=a+\|B\|$
in Theorem~\ref{th:nonneg_pert};
see \cite[Section~9.4]{birman_solomjak_book} for related considerations.

\begin{example}
Consider a Schr\"odinger operator $H_0$ in $\RR^n$ with potential $V_0$
such that $H_0$ is bounded from below and has a gap $(\alpha,\beta)$ in the spectrum.
For instance, $V_0$ can be a periodic potential.
Moreover, let $V_1$ be non-negative perturbation of $V_0$.
Let $\frh_0$ and $\frv_1$ be the quadratic forms corresponding to $H_0$ and
the multiplication operator with $V_1$ and assume that $\dom(\frh_0)\subset\dom(\frv_1)$
and that there exist $a,b\ge0$ such that
\[
  \int_{\RR^n} V_1(x)|u(x)|^2\rd x \le a\int_{\RR^n}|u(x)|^2\rd x
  + b\int_{\RR^n}\Bigl(|\nabla u(x)|^2+V_0(x)|u(x)|^2\Bigr)\rd x
\]
for $u\in\dom(\frh_0)$.
Let $H$ be the operator corresponding to the form $\frh\defeq\frh_0+\frv_1$.
If $\hat\alpha<\beta$ with $\hat\alpha$ defined as in \eqref{defalphahat},
then $(\hat\alpha,\beta)\subset\rho(H)$.
\end{example}

\section{A spectral decomposition}\label{sec:decomp}

In this section we consider operator functions that are continuous in the norm
resolvent sense and are such that, on some interval $[\alpha,\beta]$,
its spectrum consists only of a finite number of eigenvalues.
The main result is a decomposition of the space into three components:
two components are connected with the endpoints $\alpha,\beta$, and the
third component is the span of the eigenvectors corresponding to the
eigenvalues in $[\alpha,\beta]$.  This decomposition result is an analogue
of \cite[Theorem~7.3]{LMM06} where analytic operator functions whose values
are bounded operators were considered but arbitrary spectrum was allowed
in $[\alpha,\beta]$; cf.\ also similar results for Schur complements of
block operator matrices in \cite{LMMT03} and \cite{LLMT05}.
The decomposition in the following theorem is also used in the next section
to prove a variational principle.

In Proposition~\ref{pr:holom} we prove that, for holomorphic functions of type (B),
no accumulation of eigenvalues outside the essential spectrum can occur,
so that the discreteness assumption of Theorem~\ref{th:decomp} is automatically
satisfied outside the essential spectrum.


\begin{theorem}\label{th:decomp}
Let $T$ be an operator function defined on the interval $[\alpha,\beta]$, where
$\alpha,\beta\in\RR$, $\alpha<\beta$, which satisfies Assumptions {\rm(A1)--(A3)},
is continuous in the norm resolvent sense, and $T(\lambda)$ is bounded from below
for each $\lambda\in[\alpha,\beta]$.
Assume that $\alpha,\beta\in\rho(T)$ and that
\[
  \sigma(T)\cap(\alpha,\beta)=\{\lambda_1,\dots,\lambda_n\}\subset\sigmadis(T)
\]
where $\lambda_1\le\dots\le\lambda_n$ are repeated according to their multiplicities.
Moreover, let $u_1,\dots,u_n$ be corresponding linearly independent eigenvectors,
which exist by Lemma~{\rm\ref{le:06}}.  Then
\begin{equation}\label{decomp}
  \cH = \cL_{(-\infty,0)}\bigl(T(\alpha)\bigr) \dotplus \spn\{u_1,\dots,u_n\}
  \dotplus \cL_{(0,\infty)}\bigl(T(\beta)\bigr).
\end{equation}
\end{theorem}

\medskip

The next proposition gives a sufficient condition for $\sigma(T)$ having no
accumulation point on an interval.  Note that, without any further continuity assumption,
functions satisfying (A1)--(A3) may have a sequence of eigenvalues that accumulates
outside the essential spectrum; see, e.g.\ the example in Remark~\ref{re:mainths}\,(iii).
Recall that an operator function $T$ defined on a domain $U\subset\CC$
is said to be holomorphic of type (B) if $T(\lambda)$ is $m$-sectorial
for every $\lambda\in U$, the domain of the corresponding closed quadratic form
$\frt(\lambda)$ is independent of $\lambda$: $\dom(\frt(\lambda))\equiv\cD$,
and $\frt(\cdot)[x]$ is holomorphic on $U$ for every $x\in\cD$.
Instead of (A3) we assume the slightly stronger assumption:
\\[1ex]
\hspace*{0.7ex}\textbf{(A3){\boldmath$'$}} \hspace*{0.3ex} if $\frt(\lambda_0)[x]=0$ for some
$x\in\cD\setminus\{0\}$ and $\lambda_0\in\Delta$, then $\frt'(\lambda_0)[x]<0$.
\\[1ex]
In \cite{LLMT08} this condition with the reverse inequality for the derivative
was called (vm).
Note that, without any assumption of type (A3) or (A3)$'$, the result would be
incorrect as the zero function on a finite-dimensional space shows.

\begin{proposition}\label{pr:holom}
Let $U$ be a domain in $\CC$ and let $T$ be a holomorphic family of operators
of type {\rm(B)} defined on $U$.  Moreover, let $\alpha,\beta\in\RR$ with $\alpha<\beta$
be such that $(\alpha,\beta)\subset U$, $T$ satisfies Assumptions {\rm(A1), (A2), (A3)$'$}
on $(\alpha,\beta)$ and $\sess(T)\cap(\alpha,\beta)=\varnothing$.
Then $\sigma(T)\cap(\alpha,\beta)$ has no accumulation point in $(\alpha,\beta)$.
\end{proposition}

\medskip

We first prove Theorem~\ref{th:decomp}.  The main idea is to add eigenvalues
successively (see \eqref{decomp_gamma}).  The main auxiliary results needed in
this process are contained in Lemmas~\ref{le:decomp1} and \ref{le:decomp2}.

\begin{lemma}\label{le:gohberg_krupnik}
Let $\cL_1$, $\cL_2$ be closed subspaces of $\cH$ and assume that $\cL_1\cap\cL_2=\{0\}$.
The sum $\cL_1 \dotplus \cL_2$ is not closed if and only if there exist $x_n\in\cL_1$, $y_n\in\cL_2$
such that
\begin{equation}\label{GK_seq}
  \|x_n\| = 1 \quad\text{for all }n\in\NN, \qquad \|x_n+y_n\| \to 0 \quad\text{as}\quad n\to\infty.
\end{equation}
\end{lemma}

\begin{proof}
If $\cL_1\dotplus\cL_2$ is not closed, then, by \cite[Theorem~2.1.1]{GK92}, there
exist $x_n\in\cL_1$, $y_n\in\cL_2$ such that
\[
  \|x_n+y_n\| < \frac{1}{n}\Bigl(\|x_n\|+\|y_n\|\Bigr).
\]
Clearly, $x_n\ne0$ for all $n\in\NN$.  Without loss of generality we can choose $x_n$
such that $\|x_n\|=1$.  The relation
\[
  \frac{1}{n}\Bigl(\|x_n\|+\|y_n\|\Bigr) > \|x_n+y_n\| \ge \|y_n\|-\|x_n\|
\]
implies that $\|y_n\|\le \frac{n+1}{n-1}$ for $n\ge2$ and hence that $\|x_n+y_n\|\to0$, which is \eqref{GK_seq}.

Conversely, assume that there exist $x_n\in\cL_1$, $y_n\in\cL_2$ that satisfy \eqref{GK_seq}.
Then, clearly, there exists no $K$ such that
\[
  \|x_n+y_n\| \ge K\bigl(\|x_n\|+\|y_n\|\bigr) \qquad \text{for all }n\in\NN.
\]
Hence, by \cite[Theorem~2.1.1]{GK92}, the sum $\cL_1\dotplus\cL_2$ is not closed.
\end{proof}

In Lemmas~\ref{le:direct}--\ref{le:decomp2} we assume that the assumptions
of Theorem~\ref{th:decomp} are satisfied.

\begin{lemma}\label{le:direct}
Let $a,b\in[\alpha,\beta]$ with $a<b$ and
let $\cH_1\subset\cD$ be a closed subspace such that $\frt(a)[x]\le0$ for all $x\in\cH_1$.
Assume that $(0,\delta)\subset\rho(T(b))$ for some $\delta>0$.
Then the sums
\begin{equation}\label{dir_sum_closed}
  \cH_1 + \cL_{(0,\infty)}\bigl(T(b)\bigr), \qquad \cH_1 + \cL_{[0,\infty)}\bigl(T(b)\bigr)
\end{equation}
are direct and closed.
\end{lemma}

\begin{proof}
The case $\cH_1=\{0\}$ is trivial; so in the following we assume that $\cH_1\ne\{0\}$.
First observe that $\frt(b)[x]<0$ for every $x\in\cH_1\setminus\{0\}$ by Assumption~(A3).
Hence the first sum in \eqref{dir_sum_closed} is direct.  Assume that it is not closed.
Then, by Lemma~\ref{le:gohberg_krupnik}, there exist $x_n\in\cH_1$
and $y_n\in\cL_{(0,\infty)}(T(b))$, $n\in\NN$, such that
\[
  \|x_n\|=1 \quad\text{and}\quad \|x_n+y_n\|\to0 \quad\text{as}\quad n\to\infty.
\]
Set $M_0 \defeq \min \sigma(T(b))$, which is negative because $\frt(b)[x]<0$ for $x\in\cH_1\setminus\{0\}$.
Let $E$ be the spectral measure associated with the operator $T(b)$.
Then, for all $n\in\mathbb{N}$, we have
\begin{align}
  0 &> \frt(b)[x_n]
  = \int_{M_0}^0 \lambda\,\rd\bigl\langle E(\lambda)x_n,x_n\bigr\rangle
  + \int_\delta^\infty \lambda\,\rd\bigl\langle E(\lambda)x_n,x_n\bigr\rangle \notag\\[1ex]
  &\ge M_0\bigl\|E\bigl((-\infty,0)\bigr)x_n\bigr\|^2
  + \delta\bigl\|E\bigl((0,\infty)\bigr)x_n\bigr\|^2.
  \label{direct_eq1}
\end{align}
Since
\begin{align*}
  \bigl\|E\bigl((-\infty,0)\bigr)x_n\bigr\|
  &\le \bigl\|E\bigl((-\infty,0)\bigr)y_n\bigr\|
  + \bigl\|E\bigl((-\infty,0)\bigr)(x_n+y_n)\bigr\| \\[1ex]
  &= \bigl\|E\bigl((-\infty,0)\bigr)(x_n+y_n)\bigr\| \\[1ex]
  &\le \|x_n+y_n\| \to 0 \quad\text{as }n\to\infty, \\[2ex]
  \bigl\|E\bigl((0,\infty)\bigr)x_n\bigr\|
  &\ge \bigl\|E\bigl((0,\infty)\bigr)y_n\bigr\|
  - \bigl\|E\bigl((0,\infty)\bigr)(x_n+y_n)\bigr\| \\[1ex]
  &\ge \|y_n\| - \|x_n+y_n\| \to 1 \quad\text{as }n\to\infty,
\end{align*}
it follows that the right-hand side
of \eqref{direct_eq1} is positive for all sufficiently large $n$.
This contradiction shows that the first sum in \eqref{dir_sum_closed} is closed.

Since the second sum can be written as
\[
  \cH_1 + \cL_{(0,\infty)}\bigl(T(b)\bigr) + \ker\bigl(T(b)\bigr),
\]
it is closed by the first part of the proof and the fact that $\ker(T(b))$ is finite-dimensional;
see, e.g.\ \cite[Corollary~2.1.1]{GK92}.
Assume that the sum is not direct.  Then there exist $u\in\cH_1$, $v\in\cL_{(0,\infty)}(T(b))$,
$w\in\ker(T(b))$ such that $u+v+w=0$ and $v+w\ne0$.  Clearly, $\frt(b)[v+w]\ge0$.
Assumption~(A3) implies that $\frt(a)[v+w]>0$, which contradicts $\frt(a)[u]\le0$.
Hence also the second sum in \eqref{dir_sum_closed} is direct and closed.
\end{proof}

\begin{lemma}\label{le:decomp1}
Let $a,b\in[\alpha,\beta]$ be such that $a<b$ and $(a,b)\subset\rho(T)$.
Moreover, let $\cH_1\subset\cD$ be a closed subspace such that $\frt(a)[x]\le0$ for all $x\in\cH_1$.
Assume that
\begin{equation}\label{decomp11}
  \cH = \cH_1 \dotplus \cL_{(0,\infty)}\bigl(T(\mu)\bigr) \qquad\text{for all }\mu\in(a,b).
\end{equation}
Then
\begin{equation}\label{decomp12}
  \cH = \cH_1 \dotplus \cL_{[0,\infty)}\bigl(T(b)\bigr).
\end{equation}
\end{lemma}

\begin{proof}
Since $b\in\sigmadis(T)\cup\rho(T)$, the sum in \eqref{decomp12} is direct by
Lemma~\ref{le:direct} and there exists $\delta>0$ such that
$[-\delta,0)\subset\rho(T(b))$.  It follows from \cite[Theorem VI.5.10]{katopert}
that there exists $\eps>0$ such that $\bigl[-\delta,-\frac{\delta}{3}\bigr]\subset\rho(T(\mu))$
for all $\mu\in[b-\eps,b]$.  For such $\mu$ we have
$\bigl[-\frac{\delta}{3},\frac{\delta}{3}\bigr]\subset\rho\bigl(T(\mu)+\frac{2\delta}{3}\bigr)$ and
\[
  \cL_{(-\frac{2\delta}{3},\infty)}(T(\mu))
  = \cL_{(0,\infty)}\Bigl(T(\mu)+\frac{2\delta}{3}\Bigr)
  = \cL_{(0,\infty)}\biggl(\Bigl(T(\mu)+\frac{2\delta}{3}\Bigr)^{-1}\biggr).
\]
By \cite[Theorem~VII.4.2]{katopert} the operators $T(\mu)+\frac{2\delta}{3}$ are
uniformly bounded from below on $[b-\eps,b]$, say $T(\mu)+\frac{2\delta}{3}\gg M$.  Then
\[
  \sigma\biggl(\Bigl(T(\mu)+\frac{2\delta}{3}\Bigr)^{-1}\biggr)
  \subset \Bigl(-\infty,\frac{1}{M}\Bigr)\cup\Bigl[0,\frac{3}{\delta}\Bigr).
\]
If $\Gamma$ is a circle passing through $\frac1M$ and $\frac3\delta$, then
\[
  \cL_{(0,\infty)}\biggl(\Bigl(T(\mu)+\frac{2\delta}{3}\Bigr)^{-1}\biggr) = \ran P(\mu)
\]
where
\[
  P(\mu) \defeq -\frac{1}{2\pi i}\int_{\Gamma}
  \biggl(\Bigl(T(\mu)+\frac{2\delta}{3}\Bigr)^{-1}-z\biggr)^{-1}\rd z.
\]
Since $T$ is continuous in norm resolvent sense, the family of spectral projections $P(\mu)$
is uniformly continuous on the interval $[b-\eps,b]$.

Now let $x_0\in\mathcal{H}$.  We show that $x_0$ is contained in the set on the
right-hand side of \eqref{decomp12}.  To this end, let $b_n\in(b-\eps,b)$ for $n\in\NN$ with $b_n\to b$.
By \eqref{decomp11} we can write
\[
  x_0 = x_n+y_n \qquad\text{with}\quad x_n \in \cH_1,\;\;
  y_n \in \cL_{(0,\infty)}\bigl(T(b_n)\bigr).
\]
Suppose that $\|y_n\|$ is not bounded.  Without loss of generality assume that $\|y_n\|\to\infty$,
which implies that $\|x_n\|\to\infty$.  Clearly, $P(b_n)y_n=y_n$ since
$y_n\in\cL_{(0,\infty)}(T(b_n))\subset\cL_{(-\frac{2\delta}{3},\infty)}(T(b_n))$.
Set $\hat{y}_n \defeq P(b)y_n \in \cL_{[0,\infty)}(T(b))$.
Since $\delta_n \defeq \|P(b_n)-P(b)\| \to 0$ as $n\to\infty$, we have
\begin{align*}
  \left\|\frac{y_n}{\|x_n\|}-\frac{\hat{y}_n}{\|x_n\|}\right\|
  &= \frac{1}{\|x_n\|}\bigl\|\bigl(P(b_n)-P(b)\bigr)y_n\bigr\| \\[1ex]
  &\le \delta_n\frac{\|y_n\|}{\|x_n\|}
  \le \delta_n\frac{\|x_0\|+\|x_n\|}{\|x_n\|} \to 0 \qquad \text{as }n\to\infty.
\end{align*}
This relation together with $\|x_n\|\to\infty$ yields
\[
  \frac{x_n}{\|x_n\|} + \frac{\hat{y}_n}{\|x_n\|}
  = \frac{x_0}{\|x_n\|} - \left(\frac{y_n}{\|x_n\|} - \frac{\hat y_n}{\|x_n\|}\right)
  \to 0.
\]
It follows from Lemma~\ref{le:gohberg_krupnik} that $\cH_1\dotplus\cL_{[0,\infty)}(T(b))$
is not closed, which contradicts Lemma~\ref{le:direct}.  Hence the sequences $(x_n)$ and $(y_n)$
are uniformly bounded and therefore $\|y_n-\hat y_n\|\to 0$.
Setting
\[
  x_0(n) \defeq x_n + P(b)y_n \in \cH_1\dotplus\cL_{[0,\infty)}(T(b))
\]
we obtain $x_0 - x_0(n) = \bigl(P(b_n)-P(b)\bigr)y_n\to 0$.  This implies that
$x_0\in\cH_1\dotplus\cL_{[0,\infty)}(T(b))$ since the latter space is closed
by Lemma~\ref{le:direct}.
\end{proof}

\begin{lemma}\label{le:decomp2}
Let $a,\mu_0\in[\alpha,\beta)$ with $a<\mu_0$ and
let $\cH_1\subset\cD$ be a closed subspace such that $\frt(a)[x]\le0$ for all $x\in\cH_1$.
Assume that
\begin{equation}\label{decomp21}
  \cH = \cH_1 \dotplus \cL_{[0,\infty)}\bigl(T(\mu_0)\bigr).
\end{equation}
Then there exists an $\eps>0$ such that
\begin{equation}\label{decomp22}
  \cH = \cH_1 \dotplus \ker\bigl(T(\mu_0)\bigr) \dotplus \cL_{(0,\infty)}\bigl(T(\mu)\bigr)
  \quad \text{for all } \mu\in[\mu_0,\mu_0+\eps).
\end{equation}
\end{lemma}

\begin{proof}
First we prove that the sum on the right-hand side of \eqref{decomp22} is direct
and closed for all $\mu\in[\mu_0,\beta]$.  It follows from Lemma~\ref{le:direct}
that the sum $\cH_1+\cL_{(0,\infty)}(T(\mu))$ is direct and closed.
Since $\ker(T(\mu_0))$ is finite-dimensional, the sum on the right-hand side
of \eqref{decomp22} is closed; see \cite[Corollary~2.1.1]{GK92}.
Assume that it is not direct.  Then there exist $u\in\cH_1$, $v\in\ker(T(\mu_0))$,
$w\in\cL_{(0,\infty)}(T(\mu))$ such that $u+v+w=0$ and $w\ne0$.
By Lemma~\ref{le:05}\,(ii) we have $\frt(\mu)[u+v]\le0$, which contradicts
$w\in\cL_{(0,\infty)}(T(\mu))$.  Hence the sum on the right-hand side of \eqref{decomp22}
is direct and closed.

Next we show that there exists a $K>0$ such that
\begin{equation}\label{boundK}
\begin{aligned}
  & x\in\cH_1,\;y\in\ker\bigl(T(\mu_0)\bigr),\;w\in\cL_{(0,\infty)}\bigl(T(\mu_0)\bigr),\;\;
  \|x+y+w\|=1 \\[1ex]
  &\Longrightarrow\quad \|w\|\le K.
\end{aligned}
\end{equation}
Assume that this is not true.  Then there exist $x_n\in\cH_1$, $y_n\in\ker(T(\mu_0))$,
$w_n\in\cL_{(0,\infty)}(T(\mu_0))$ such that $\|x_n+y_n+w_n\|=1$ and $\|w_n\|\to\infty$.
In this case also $\|y_n+w_n\|\to\infty$ and hence $\|x_n\|\to\infty$.
Since
\[
  \frac{x_n}{\|x_n\|}+\frac{y_n+w_n}{\|x_n\|}
  = \frac{x_n+y_n+w_n}{\|x_n\|} \to 0,
\]
Lemma~\ref{le:gohberg_krupnik} implies that the sum $\cH_1 \dotplus \cL_{[0,\infty)}(T(\mu_0))$
is not closed, which contradicts \eqref{decomp21}.
Hence a $K>0$ with the desired property exists.

Let $P(\mu)$ be the orthogonal projection onto $\cL_{(0,\infty)}(T(\mu))$ for $\mu\in[\mu_0,\beta]$.
Similarly as in the proof of the previous lemma one shows that
$d_\mu \defeq \|P(\mu)-P(\mu_0)\| \to 0$ as $\mu\searrow\mu_0$.  Hence
there exists an $\eps>0$ such that $\delta_\mu K<1$ for all $\mu\in[\mu_0,\mu_0+\eps)$.
We show that \eqref{decomp22} holds for all such $\mu$.
Assume that this is not the case.  Then, for some $\mu\in[\mu_0,\mu_0+\eps)$ there exists an $x_0\in\cH$
with $\|x_0\|=1$ which is orthogonal to the right-hand side of \eqref{decomp22}.
Since \eqref{decomp21} is true by assumption, we can write
\[
  x_0 = u+v+w \qquad\text{with}\quad
  u\in\cH_1,\; v\in\ker\bigl(T(\mu_0)\bigr),\;w\in\cL_{(0,\infty)}\bigl(T(\mu_0)\bigr).
\]
By \eqref{boundK} we have $\|w\|\le K$.  Now set $y \defeq u+v+P(\mu)w$, which is contained
in the right-hand side of \eqref{decomp22}.  Then
\[
  \|x_0-y\| = \|w-P(\mu)w\| = \bigl\|\bigl(P(\mu_0)-P(\mu)\bigr)w\bigr\|
  \le \delta_\mu K < 1,
\]
which is a contradiction to the facts that $x_0\perp y$ and $\|x_0\|=1$.
\end{proof}

\begin{proof}[Proof of Theorem~\ref{th:decomp}]
Let $\hat\lambda_1<\cdots<\hat\lambda_m$ be the eigenvalues of $T$ in the interval $(\alpha,\beta)$
\emph{not} counted with multiplicities and set $\hat\lambda_0 \defeq \alpha$.
For $\gamma\in[\alpha,\beta]$ consider the statement
\begin{equation}\label{decomp_gamma}
  \hspace*{-1ex}\begin{aligned}
    \cH = \cL_{(-\infty,0)}\bigl(T(\alpha)\bigr) \dotplus \ker\bigl(T(\hat\lambda_1)\bigr)
    \dotplus \ldots \dotplus \ker\bigl(T(\hat\lambda_k)\bigr) \dotplus \cL_{[0,\infty)}\bigl(T(\gamma)\bigr) & \\[0.5ex]
    \text{where $k$ is such that } \hat\lambda_1,\dots,\hat\lambda_k \text{ are the eigenvalues of $T$ in } (\alpha,\gamma). &
  \end{aligned}
\end{equation}
If $(\alpha,\gamma)$ contains no eigenvalues, then $k=0$.
For $\gamma=\alpha$ the statement is certainly true.
We prove that \eqref{decomp_gamma} holds for all $\gamma\in[\alpha,\beta]$.
Assume that this is not the case and let
$\gamma_0 \defeq \inf\bigl\{\gamma\in[\alpha,\beta]:
\text{\eqref{decomp_gamma} does not hold}\bigr\}$.
Set
\[
  \cH_1 \defeq \cL_{(-\infty,0)}\bigl(T(\alpha)\bigr) \dotplus \ker\bigl(T(\hat\lambda_1)\bigr)
  \dotplus \ldots \dotplus \ker\bigl(T(\hat\lambda_k)\bigr)
\]
where $k$ is such that $\hat\lambda_1,\dots,\hat\lambda_k$ are the eigenvalues of $T$
in the interval $(\alpha,\gamma_0)$.
Lemma~\ref{le:05}\,(ii) implies that $\frt(\hat\lambda_k)[x]\le0$ for all $x\in\cH_1$.

It follows from Lemma~\ref{le:decomp1} with $a=\hat\lambda_k$ and $b=\gamma_0$ that \eqref{decomp_gamma}
holds also for $\gamma=\gamma_0$.  Now, if $\gamma_0<\beta$, then Lemma~\ref{le:decomp2} yields a
contradiction with the definition of $\gamma_0$.
Hence \eqref{decomp_gamma} holds for all $\gamma\in[\alpha,\beta]$.  For $\gamma=\beta$ this is exactly
the assertion of the theorem.
\end{proof}

In order to prove Proposition~\ref{pr:holom}, we first need the following lemma.

\begin{lemma}\label{le:der_form_holom}
Let $T$ be a holomorphic family of operators of type {\rm(B)} defined
on the complex domain $U\subset\CC$ with closed forms $\frt$ such
that $\dom(\frt(\lambda))=\cD$ for all $\lambda\in U$.
Moreover, let $x(\lambda)\in\cD$ for $\lambda\in U$ such that $x(\cdot)$
is holomorphic.
\begin{myenum}
\item
Assume that $\frt(\lambda)[x(\lambda)]$ is locally bounded
and let $y_0\in\cD$.  Then $\frt(\lambda)[x(\lambda),y_0]$ is holomorphic
in $\lambda$, $x'(\lambda)\in\cD$ and
\begin{equation}\label{prod_rule}
  \frac{\rd}{\rd\lambda}\Bigl(\frt(\lambda)[x(\lambda),y_0]\Bigr)
  = \frt'(\lambda)[x(\lambda),y_0]+\frt(\lambda)[x'(\lambda),y_0].
\end{equation}
for all $\lambda\in U$.
\item
Assume that $T(\lambda)x(\lambda)=\nu(\lambda)x(\lambda)$ where $\nu$
is a scalar holomorphic function on $U$.  Further, let $\lambda_0\in U$
and assume that there exists a $y_0\in\dom(T(\lambda_0)^*)$ such that
$T(\lambda_0)^*y_0=\ov{\nu(\lambda_0)}y_0$ and $\langle x(\lambda_0),y_0\rangle\ne0$.
Then
\begin{equation}\label{der_nu}
  \nu'(\lambda_0) = \frac{\frt'(\lambda_0)\bigl[x(\lambda_0),y_0\bigr]}{\bigl\langle x(\lambda_0),y_0\bigr\rangle}\,.
\end{equation}
\end{myenum}
\end{lemma}

\medskip

\noindent
Item (ii) of this lemma can be applied, in particular, if $T(\lambda_0)$ is self-adjoint
and one chooses $y_0=x(\lambda_0)$.

\begin{proof}
(i)
Fix $\lambda_0\in U$ and choose $M\in\RR$ such that
$\Re\frt(\lambda_0)+M\gg0$. According to \cite[(VII.4.4)]{katopert}
the form $\frt$ can be written as
\[
  \frt(\lambda)[u,v] = \bigl\langle T_0(\lambda)Gu,Gv\bigr\rangle
  - M\langle u,v\rangle, \qquad u,v\in\cD,
\]
where $T_0$ is a holomorphic operator function whose values are bounded
operators and $G\defeq (\Re T(\lambda_0)+M)^{1/2}$.

Now set $y(\lambda)\defeq Gx(\lambda)$ for $\lambda\in U$.
It follows from \cite[(VII.4.7)]{katopert} that, for each compact subset
$U_0$ of $U$ with $\lambda_0\in U_0$ there exists $C>0$ such that
\begin{align*}
  \|y(\lambda)\|^2 &= \bigl\langle Gx(\lambda),Gx(\lambda)\bigr\rangle
  = \Re\frt(\lambda_0)[x(\lambda)]+M\|x(\lambda)\|^2 \\
  &\le C\bigl|\frt(\lambda)[x(\lambda)]\bigr|+M\|x(\lambda)\|^2
\end{align*}
for all $\lambda\in U_0$.  Since the last expression is bounded on $U_0$
by assumption, it follows that $y(\lambda)$ is locally bounded.
For $u\in\cD$, the scalar function
\[
  \bigl\langle y(\lambda),u\bigr\rangle = \bigl\langle Gx(\lambda),u\bigr\rangle
  = \bigl\langle x(\lambda),Gu\bigr\rangle
\]
is holomorphic in $\lambda$.  Hence $y(\lambda)$ is strongly holomorphic
in $\lambda$; see, e.g.\ \cite[\S VII.1.1]{katopert}.
Moreover, $\langle y'(\lambda),u\rangle = \langle x'(\lambda),Gu\rangle$
for all $u\in\cD=\dom G$, which implies that $x'(\lambda)\in\dom G=\cD$
and $y'(\lambda)=Gx'(\lambda)$.

We conclude that the function
\[
  \frt(\lambda)[x(\lambda),y_0]
  = \bigl\langle T_0(\lambda)y(\lambda),Gy_0\bigr\rangle
  - M\bigl\langle x(\lambda),y_0\bigr\rangle
\]
is holomorphic and that
\begin{align*}
  &\frac{\rd}{\rd\lambda}\Bigl(\frt(\lambda)[x(\lambda),y_0]\Bigr)
  = \frac{\rd}{\rd\lambda}\Bigl(\bigl\langle T_0(\lambda)y(\lambda),Gy_0\bigr\rangle
  -M\bigl\langle x(\lambda),y_0\bigr\rangle\Bigr) \\[1ex]
  &= \bigl\langle T_0'(\lambda)y(\lambda),Gy_0\bigr\rangle
  + \bigl\langle T_0(\lambda)y'(\lambda),Gy_0\bigr\rangle
  - M\bigl\langle x'(\lambda),y_0\bigr\rangle \\[1ex]
  &= \bigl\langle T_0'(\lambda)Gx(\lambda),Gy_0\bigr\rangle
  + \bigl\langle T_0(\lambda)Gx'(\lambda),Gy_0\bigr\rangle
  - M\bigl\langle x'(\lambda),y_0\bigr\rangle \\[1ex]
  &= \frt'(\lambda)[x(\lambda),y_0] + \frt(\lambda)[x'(\lambda),y_0],
\end{align*}
which shows \eqref{prod_rule}.

(ii)
The expression $\frt(\lambda)[x(\lambda)] = \nu(\lambda)\|x(\lambda)\|^2$
is locally bounded in $\lambda$.  Hence we can apply item (i) of this lemma
to the derivative of the equality $\frt(\lambda)[x(\lambda),y_0]=\nu(\lambda)\langle x(\lambda),y_0\rangle$,
which, for $\lambda=\lambda_0$, yields
\begin{equation}\label{382}
  \frt'(\lambda_0)\bigl[x(\lambda_0),y_0\bigr] + \frt(\lambda_0)\bigl[x'(\lambda_0),y_0\bigr]
  = \nu'(\lambda_0)\bigl\langle x(\lambda_0),y_0\bigr\rangle + \nu(\lambda_0)\bigl\langle x'(\lambda_0),y_0\bigr\rangle.
\end{equation}
The second term on the left-hand side is equal to
\[
  \bigl\langle x'(\lambda_0),T(\lambda_0)^*y_0\bigr\rangle
  = \bigl\langle x'(\lambda_0),\ov{\nu(\lambda_0)}y_0\bigr\rangle
  = \nu(\lambda_0)\bigl\langle x'(\lambda_0),y_0\bigr\rangle.
\]
Hence \eqref{382} yields the desired result.
\end{proof}

\begin{proof}[Proof of Proposition~\ref{pr:holom}]
Let $\mu\in(\alpha,\beta)$.  The assumption $\sess(T)\cap(\alpha,\beta)=\varnothing$
implies that there exists an $\eps>0$ such that $\sigma(T(\mu))\cap(-\eps,\eps)\subset\{0\}$
and $n\defeq\dim\ker(T(\mu))$ is finite.  Further, there exists a $\delta_\eps>0$
such that, for all $\lambda$ with $|\lambda -\mu|<\delta_\eps$ the intersection
$\sigma(T(\lambda))\cap(\eps,\eps)$ consists only of eigenvalues of finite multiplicity
with total multiplicity $n$.  These eigenvalues can be enumerated such that they are
analytic functions of $\lambda$.
Let $\nu(\lambda)$ be such an eigenvalue curve with a zero at $\lambda_0$
(i.e.\ $\lambda_0$ is an eigenvalue of $T$), extend it also to a complex
neighbourhood of $\lambda_0$ and let $x(\lambda)$ be corresponding eigenvectors,
which can be chosen to depend analytically on $\lambda$;
see \cite[\S \S VII.6.2 and II.6.2]{katopert}.

Now we can apply Lemma~\ref{le:der_form_holom}\,(ii) with $y_0=x(\lambda_0)$, which yields
\[
  \nu'(\lambda_0) = \frac{\frt'(\lambda_0)[x(\lambda_0)]}{\|x(\lambda_0)\|^2}\,.
\]
Since, by Assumption (A3)$'$, this expression is negative, eigenvalue curves
can cross the $\lambda$-axis only in one direction.  Therefore $T$ has at most $N$
eigenvalues in $[\mu-\delta,\mu+\delta]$,
and hence the eigenvalues cannot accumulate at $\mu$.
\end{proof}

\section{Variational principles for norm resolvent continuous operator functions}\label{sec:var_equ}

In this section we prove that under stronger continuity assumptions on the
operator function we have equality in the variational principle from
Theorem~\ref{th:varinequ}.

\begin{theorem} \label{th:var_equ}
Let $\Delta\subset\RR$ be an interval with right endpoint $\beta\in\RR\cup\{\infty\}$
and let $T$ be an operator function defined on $\Delta$ which satisfies
Assumptions {\rm(A1)--(A3)} on $\Delta$, is continuous in the
norm resolvent sense on $\Delta$, and $T(\lambda)$ is bounded from below
for each $\lambda\in\Delta$.
Moreover, let $p$ be a generalised Rayleigh functional for $T$ on $\Delta$,
let $\gamma\in\rho(T)\cap\Delta$ with $\gamma<\beta$, let $\M_\gamma^+$ be defined as
in Definition~\ref{def:rayleigh_Mg+} and let $\lae$ be as in \eqref{deflae}.

Assume that the spectrum of\, $T$ in $(\gamma,\lae)$ has no accumulation point in
$[\gamma,\lae)$, i.e.\ $\sigma(T)\cap[\gamma,\lae)$ is empty or consists of a finite or infinite
non-decreasing sequence of eigenvalues $(\lambda_n)_{n=1}^N$ with $N\in\NN\cup\{\infty\}$,
counted according to their multiplicities, which can accumulate at most at $\lae$.

If $\sigma(T)\cap(\gamma,\lae)\ne\varnothing$, then
\begin{equation} \label{varequ_holom}
  \lambda_n = \sup_{\cM\in\M_\gamma^+} \sup_{\substack{\cL\subset\cM \\[0.2ex] \dim\cL=n-1}}
  \inf_{\substack{x\in\cM\setminus\{0\} \\[0.2ex] x\perp\cL}}\; p(x),
  \qquad n\in\NN,\,n\le N.
\end{equation}
Moreover, if, in addition, $T$ satisfies the condition \VM, $N$ is finite
and $\sess(T)\cap(\gamma,\beta)\ne\varnothing$, then
\begin{equation} \label{var_lae_holom}
  \lae = \sup_{\cM\in\M_\gamma^+} \sup_{\substack{\cL\subset\cM \\[0.2ex] \dim\cL=n-1}}
  \inf_{\substack{x\in\cM\setminus\{0\} \\[0.2ex] x\perp\cL}}\; p(x),
  \qquad n > N.
\end{equation}
\end{theorem}

\medskip

\begin{remark}\label{re:triple_holom}
If $T$ is a holomorphic family of type (B) in a neighbourhood of $\Delta$
and Assumption (A3)$'$ is satisfied, then
one does not have to assume that the eigenvalues cannot accumulate in $[\gamma,\lae)$,
but this follows from Proposition~\ref{pr:holom}.
Theorem~\ref{th:var_equ} and Proposition~\ref{pr:holom} can be applied, e.g.\ to
operator polynomials and Schur complements of certain block operator matrices;
for the latter see \cite{LS_bom}.
\end{remark}

\begin{proof}
The inequalities `$\ge$' in \eqref{varequ_holom} and \eqref{var_lae_holom} follow from
Theorem~\ref{th:varinequ}.  We first prove `$\le$' in \eqref{varequ_holom}. Let $1\le n \le N$.
It is sufficient to find a subspace $\cM\in\M_\gamma^+$
and a subspace $\cL\subset\cM$ with $\dim\cL=n-1$ such that
\begin{equation}\label{T115}
  \inf_{\substack{x\in\cM\setminus\{0\} \\[0.2ex] x\perp\cL}} p(x)
  \ge \lambda_n.
\end{equation}
Let $m=\max\{k\in\NN:\lambda_k=\lambda_n\}$ and choose $\mu>\lambda_n$ such
that $(\lambda_n,\mu]\subset\rho(T)$.
Moreover, let $u_1,\dots,u_m$ be linearly independent eigenvectors of $T$
corresponding to the eigenvalues $\lambda_1,\dots\lambda_m$ (see Lemma~\ref{le:06}).
Consider the subspace
\[
  \cM \defeq \spn\{u_1,\dots,u_m\}+\bigl(\mathcal{L}_{(0,\infty)}(T(\mu))\cap\cD\bigr).
\]
From Lemma~\ref{le:05}\,(i) we obtain that $\frt(\gamma)[x]\ge0$ for all $x\in\mathcal{M}$.
Since $\cL_{(-\infty,0)}(T(\gamma))\subset\cD$ and $u_k\in\cD$, $k=1,\dots,m$,
Theorem~\ref{th:decomp} implies that the following decomposition of $\cD$ is valid:
\begin{equation}\label{decomp_D}
  \cD = \cL_{(-\infty,0)}\bigl(T(\gamma)\bigr) \dotplus \spn\{u_1,\dots,u_m\}
  \dotplus \bigl(\cL_{(0,\infty)}(T(\mu))\cap\cD\bigr).
\end{equation}
It follows from this decomposition that $\cM$ is maximal $\frt(\gamma)$-non-negative,
i.e.\ $\cM\in\M_\gamma^+$.
Let $P$ be the orthogonal projection in $\cH$ onto
\[
  \cK \defeq \spn\{u_n,\dots,u_m\}+\cL_{(0,\infty)}\bigl(T(\mu)\bigr)
\]
and set
\[
  \cL \defeq (I-P)\cM = (I-P)\spn\{u_1,\dots,u_{n-1}\}.
\]
Since
\[
  \ran(I-P) = \cK^\perp \subset \cL_{(-\infty,0)}\bigl(T(\mu)\bigr)
  \subset \dom\bigl(T(\mu)\bigr) \subset \cD
\]
and
\[
  (I-P)u_k = u_k-Pu_k \in \cM+\cK, \qquad k=1,\dots,n-1,
\]
we have $\cL \subset \cD\cap(\cM+\cK) = \cM$.
We show that the mapping $I-P$ is injective on $\spn\{u_1,\dots,u_{n-1}\}$.
Assume that this is not the case.  Then there exists a $u\in\spn\{u_1,\dots,u_{n-1}\}$,
$u\ne0$, such that $u\in\ker(I-P)=\cK$, i.e.\ there exist $\alpha_1,\ldots,\alpha_m\in\CC$
and $x\in\cL_{(0,\infty)}(T(\mu))\cap\cD$ such that
\[
  u = \alpha_1u_1+\ldots+\alpha_{n-1}u_{n-1} = \alpha_nu_n+\ldots+\alpha_mu_m+x.
\]
Since the sum in \eqref{decomp_D} is direct, we have $x=0$, and therefore
$\alpha_1=\ldots=\alpha_m=0$ because of the linear independence of $u_1,\dots,u_m$.
Hence $I-P$ is injective on $\spn\{u_1,\dots,u_{n-1}\}$, which shows that
$\dim\cL=n-1$.

Now let $x\in\cM\setminus\{0\}$ such that $x\perp\cL=(I-P)\cM$.  Then $x\in\cD$,
and the relation $x=Px+(I-P)x$ implies that
\[
  \|x\|^2 = \langle Px,x\rangle + \langle(I-P)x,x\rangle = \|Px\|^2,
\]
which shows that $x\in\ran P=\cK$.  It follows from Lemma~\ref{le:05}\,(i)
that $\frt(\lambda_n)[x]\ge0$, which proves \eqref{T115}.

In order to prove \eqref{var_lae_holom}, assume that $N$ is finite and let $n>N$.
Moreover, let $\mu\in(\lambda_N,\lae)$ be arbitrary and $P$ be the orthogonal projection
in $\cH$ onto $\cL_{(0,\infty)}(T(\mu))$.
Similarly to the first part of the proof we can choose
\[
  \cM \defeq \spn\{u_1,\dots,u_N\}+\bigl(\cL_{(0,\infty)}(T(\mu))\cap\cD\bigr),
\]
which is in $\M_\gamma^+$.  The space $\cL'\defeq(I-P)\cM$ is an
$N$-dimensional subspace of $\cM$, which can be seen as above.  Extend $\cL'$ to
an $(n-1)$-dimensional subspace $\cL$ of $\cM$.  Then
\[
  \inf_{\substack{x\in\cM\setminus\{0\} \\[0.2ex] x\perp\cL}} p(x) \ge \mu,
\]
which shows \eqref{var_lae_holom} since $\mu\in(\lambda_N,\lae)$ was arbitrary.
\end{proof}

\begin{example}
Let $C$ be a self-adjoint operator in a Hilbert space $\cH$ that is bounded
from below, let $\frc$ be the corresponding quadratic form,
and assume that $0\in\rho(C)$.
Moreover, let $\frb$ be a symmetric non-positive quadratic form
that is $\frc$-bounded with relative bound $0$, i.e.\
$\dom(\frc)\subset\dom(\frb)$ and for each $b>0$ there exists an $a\ge0$
such that
\[
  \bigl|\frb[x]\bigr| \le a\|x\|^2 + b\bigl|\frc[x]\bigr|, \qquad x\in\dom(\frc).
\]
For instance, $\frb$ can be a form corresponding to a $C$-compact operator.
Then the form
\[
  \frt(\lambda)[x,y] \defeq -\lambda^2\langle x,y\rangle + \lambda\frb[x,y] + \frc[x,y],
  \qquad x,y\in\dom(\frc),
\]
is sectorial and closed for every $\lambda\in\CC$ and hence defines an m-sectorial
operator $T(\lambda)$.
Clearly, $T$ is a holomorphic family of type (B) on $\CC$
and Assumptions (A1) and (A2) are satisfied on $\Delta\defeq[0,\infty)$.

For $x\in\cD\setminus\{0\}$ denote by $p_\pm(x)$ the solutions of
$\frt(\lambda)[x]=0$
with $p_-(x)\le p_+(x)$ if the solutions are real, and set $p_\pm(x)\defeq-\infty$ otherwise.
Since $\frb[x]\le0$, we have $p_-(x)\le0$ for $x\in\cD\setminus\{0\}$,
Assumption (A3) is satisfied
and $p_+$ is a generalised Rayleigh functional for $T$ on $\Delta$.
Hence we can apply Theorem~\ref{th:var_equ} together with Proposition~\ref{pr:holom}
(see Remark~\ref{re:triple_holom}) with $\gamma=0$, which yields a
characterisation of the eigenvalues $\lambda_1\le\lambda_2\le\cdots$
in the interval $(0,\min(\sess(T)\cap(0,\infty)))$, given by the
formula in \eqref{varequ_holom} with $p_+$ instead of $p$ and $\gamma=0$.

We can compare these eigenvalues with the eigenvalues of the
operator polynomial $T_0(\lambda) = -\lambda^2 + C$, which satisfies also
all assumptions of Theorem~\ref{th:var_equ}.
The corresponding generalised Rayleigh functional
is $\mathring p_+(x)=\sqrt{\frc[x]}/\|x\|$ if $\frc[x]\ge0$.
Since $T(0)=T_0(0)$, the maximal non-negative subspaces are the same for $T$
and $T_0$ at $\gamma=0$.  Denote by $\mu_1\le\mu_2\le\cdots$ the eigenvalues
of $C$ in the interval $(0,\min(\sess(C)\cap(0,\infty)))$.
Since $p_+(x) \le \mathring p_+(x)$, we obtain the
inequalities $\lambda_n \le \sqrt{\mu_n}$.

\end{example}

\section*{Acknowledgements}
Both authors gratefully acknowledge the support of the
Engineering and Physical Sciences Research Council (EPSRC), grant no.\
EP/E037844/1. \\
M.~Strauss gratefully acknowledges the support from the
Wales Institute of Mathematical and Computational Sciences
and the Leverhulme Trust, grant no.\ RPG-167.


\end{document}